\documentclass[11pt, reqno]{amsart}
\usepackage{amscd,amsfonts,amsmath,amssymb,amstext,amsthm}
\usepackage[all]{xy}
\usepackage{xcolor}
\usepackage{enumerate}
\usepackage{verbatim}
\usepackage{enumitem}
\usepackage[linktocpage=true]{hyperref}

\makeatletter

\newcommand{\Rmnum}[1]{\expandafter\@slowromancap\romannumeral #1@}
\makeatother

\theoremstyle{plain}
\newtheorem{theorem} {Theorem} [section]
\newtheorem{lemma} [theorem]{Lemma}
\newtheorem{proposition}[theorem]{Proposition}
\newtheorem{corollary} [theorem]{Corollary}
\theoremstyle{definition}
\newtheorem{definition} [theorem] {Definition}
\newtheorem{example}[theorem] {Example}

\numberwithin{equation}{section}

\title{Symmetric differentials and jets extension of $L^2$ holomorphic functions}
\date\today

\author{Seungjae Lee and Aeryeong Seo}

\address{Center for Complex Geometry, Institute for Basic Science, 55, 
Expo-ro, Yuseong-gu, Daejeon, 34126, South Korea}
\email{seungjae@ibs.re.kr}

\address{Department of Mathematics,
Kyungpook National University,
Daegu 41566, South Korea}%
\email{aeryeong.seo@knu.ac.kr}

\subjclass[2010]{Primary 32A36,
Secondary 32A05, 32A25, 32N99,
32T27, 32W05, 32Q05, 32L10.}%
\keywords{complex hyperbolic space forms, symmetric differentials, $L^2$ holomorphic functions, $\bar\partial$-equations}

\begin{document}

\maketitle
\markboth{ Seungjae Lee, Aeryeong Seo}{Symmetric differentials and $L^2$ holomorphic functions}

\begin{abstract}
Let $\Sigma = \mathbb B^n/\Gamma$ be a complex hyperbolic space with discrete  subgroup $\Gamma$ of the automorphism group of the unit ball $\mathbb B^n$ and $\Omega $ be a quotient of $\mathbb B^n \times\mathbb B^n$ under
the diagonal action of $\Gamma$ which is a holomorphic $\mathbb B^n$-fiber bundle over $\Sigma$.
The goal of this article is to investigate the relation between symmetric differentials of $\Sigma$
and the weighted $L^2$ holomorphic functions of $\Omega$.
If there exists a holomorphic function on $\Omega$  and it vanishes up to $k$-th order on the maximal compact complex variety in $\Omega$, then
there exists a symmetric differential of degree $k+1$ on $\Sigma$.
Using this property, we show that $\Sigma$ always has a symmetric differential of degree $N$ for any $N \geq n+2$.
Moreover if $\Sigma$ is compact, for each symmetric differential over $\Sigma$
we construct a weighted $L^2$ holomorphic function on $\Omega$. 
We also show that any bounded holomorphic function on $\Omega$ is constant when 
$H^0 (\Sigma, S^{m} T_{\Sigma}^* )=0$ for every $0 < m \leq n+1$.
\end{abstract}

\section{Introduction}
Let $\Sigma $ be a complex hyperbolic space form, i.e. $\Sigma = \mathbb B^n /\Gamma$
 is the quotient of the complex unit ball $\mathbb B^n= \{ z\in \mathbb C^n : |z|<1 \}$ by a discrete torsion-free subgroup $\Gamma$ of the automorphism group $\text{Aut}(\mathbb B^n)$ of $\mathbb B^n$. Since the holomorphic cotangent bundle $T^*_\Sigma$
of $\Sigma$ is ample, its $m$-th symmetric power $S^mT^*_\Sigma$ is generated by global sections when $m$ is sufficiently large.

In general the existence of such global holomorphic sections, also known as symmetric differentials,
is expected to be related to the topological fundamental group of the base manifold.
In the case Riemann surfaces, the existence of symmetric differentials is related to the degree of the cotangent bundle.
In particular, when the cotangent
bundle has positive degree, i.e. when the genus $q$ of the curve is greater than or equal to two, the dimension of
the set of symmetric differentials of degree $m$ equals to $(q-1)(2m-1)$.
However if we consider higher dimensional complex manifolds,
relation between topological properties and the symmetric differentials is not simple.
We refer the interested reader to \cite{Brunebarbe-Klingler-Totaro, Klingler} and the references therein.

Consider $\Omega := \mathbb B^n \times \mathbb B^n /\Gamma$ and  $\mathbb B^n\times \mathbb C\mathbb P^n/\Gamma$ under the diagonal action
$$
	\gamma(z,w) := (\gamma z, \gamma w)
$$
with $\gamma \in \Gamma$. Since there is a canonical embedding of $\text{Aut}(\mathbb B^n)$ into
$\text{Aut}(\mathbb C\mathbb P^n)$, the diagonal action is well defined on $\mathbb B^n\times \mathbb C\mathbb P^n$.
Let $D$ be the maximal compact complex variety in $\Omega$, i.e. $D = \{(z,z) \in \mathbb B^n\times\mathbb B^n \} / \Gamma$.
Inside $\mathbb B^n\times \mathbb C\mathbb P^n/\Gamma$, the domain $\Omega$ is connected with real analytic boundary. Moreover it possesses a bounded plurisubharmonic exhaustion function $-\rho^{1 \over 2}$ where $\rho$ is given by
$$
\rho(z,w) = \frac{(1-|z|^2)(1-|w|^2)}{|1-w \cdot \bar z|^2}
$$
(cf. \cite{Seo}). Since it is invariant under the diagonal action, $\rho$ is well defined on $\Omega$.
Note that one may consider $\Omega$ as a holomorphic $\mathbb B^n$-fiber bundle over $\Sigma$.
The case $n=1$ was recently studied by Adachi.
In \cite{Adachi}, the set of weighted $L^2$ holomorphic functions on $\Omega$ was identified by determining  how jets of holomorphic functions belonging to
$H^0(D, \mathcal I_D^N / \mathcal I_D^{N+1})$ extend to holomorphic functions on the whole of $\Omega$.
Here $\mathcal I_D$ denotes the ideal sheaf of $D$.
Starting from a holomorphic section of the $k$-th tensor power of the canonical line bundle of the Riemann surface, he constructed a sequence of recursive $\bar\partial$-equations to obtain the higher order terms.

The aim of this article is to generalize the results of Adachi to
the case of higher dimensional hyperbolic space forms. We describe the symmetric differentials on $\Sigma$ and weighted $L^2$ holomorphic functions on  $\Omega$.

Let $\mathcal I_{D}^k$
denote the subset of
$\mathcal O(\Omega)$ whose elements vanish on $D$ up to the $k$-th order.
Our main results are the following:

\begin{theorem}\label{SD}
Let $\Sigma = \mathbb B^n / \Gamma$ be a complex hyperbolic space form with a torsion-free discrete subgroup of the automorphism group of the complex unit ball $\mathbb B^n$.
Let $\Omega = \mathbb B^n \times \mathbb B^n /\Gamma$ be a holomorphic $\mathbb B^n$-fiber bundle
under the diagonal action of $\Gamma$ on $\mathbb B^n\times\mathbb B^n$
and let $D$ be the subset $\{[(z,z)]\in \Omega: z\in \mathbb B^n\}$.
Then there exists a map
$$
	\Psi \colon \mathcal O(\Omega)
	\rightarrow \bigoplus_{m=0}^\infty H^0(\Sigma, S^mT^*_\Sigma)
$$
such that if $f\in \mathcal O(\Omega)$ vanishes on $D$ up to $k$-th order,
then $\Psi(f)$ is a symmetric differential of degree $k+1$.
\end{theorem}
In particular, since we have Poincar\'e
series on $\mathbb B^n\times\mathbb B^n$, we have the following.
\begin{corollary}\label{SD_cor}
Let $\Gamma$ be a torsion-free discrete subgroup of the automorphism group of the unit ball $\mathbb B^n$.
For any $N\geq n+2$, there exists a symmetric differential of degree $N$ over $\mathbb B^n/\Gamma$.
\end{corollary}

\begin{theorem}\label{main}
Let $\Sigma = \mathbb B^n / \Gamma$ be a compact complex hyperbolic space form with a torsion-free cocompact lattice of the automorphism group of the complex unit ball $\mathbb B^n$.
Let $\Omega = \mathbb B^n \times \mathbb B^n /\Gamma$ be a holomorphic $\mathbb B^n$-fiber bundle
under the diagonal action of $\Gamma$ on $\mathbb B^n\times\mathbb B^n$.
Then we have an injective linear map
$$
\Phi\colon \bigoplus_{m=n+2}^\infty H^0(\Sigma, S^m T^*_\Sigma)
\rightarrow \bigcap_{\alpha>-1}
A^2_{\alpha}(\Omega) \subset \mathcal O(\Omega)
$$
having a dense image in $\mathcal O(\Omega)/\mathcal I_{D}^{n+1}$ equipped with
compact open topology on $\mathcal O(\Omega)$.
\end{theorem}

The map $\Phi$ defined in Theorem~\ref{main}
is constructed as follows.
Let $N \geq n+2$. For a given $\psi \in H^{0}(\Sigma, S^{N}T_{\Sigma}^{*})$,
define a sequence $\varphi_k\in C^\infty(\Sigma, S^{N+m}T_{\Sigma}^{*})$ with $k=1,2, \ldots $ given by
\begin{equation*}
\left\{ \begin{array}{ll}
\varphi_{k}=0 & \text{if $k<N$}, \\
\varphi_{N}=\psi ,&
\end{array} \right.
\end{equation*}
and for $m \geq 1$,  the section $\varphi_{N+m}$ is the $L^2$ minimal solution of
the following $\overline \partial$-equation:
\begin{equation*}
\bar \partial \varphi_{N+m} = - (N+m -1) \mathcal R_G\left( \varphi_{N+m-1}  \right),
\end{equation*}
where $\mathcal R_G$ is the raising operator given in Section \ref{raising operator} with respect to
the K\"ahler form $G$ of the induced metric from the Bergman metric on $\mathbb B^n$.
Denote $\varphi_m = \sum_{|I|=m}f_I e^I$ with a unitary frame $e$ of $T^*_\Sigma$
given by \eqref{e} and define
$$
\Phi(\psi) (z,w):= \sum_{|I|=0}^\infty f_I(z) (T_zw)^I
$$ where
$T_z$ is an automorphism on $\mathbb B^n$ so that $T_zz=0$ given in \eqref{Tzw}.

The proof of Theorem \ref{main} relies crucially on the Hodge type identities for raising operators given in Section~\ref{raising operator}.
In the case of a compact K\"ahler manifold $X$ with K\"ahler form $G$, there exists a collection of identities satisfied by the operators defined on $\Lambda^{p,q}T^*_X$, the holomorphic vector bundle of exterior differentials. These play an important role in Hodge theory.
One of them is given by $[L, \Box]=0$ where $L$ is the Lefschetz operator given by $L : u \mapsto u\wedge G$,  $u\in \Lambda^{p,q}T^*_X$ and $\Box = \bar\partial \bar\partial^*+\bar\partial^*\bar\partial$ is the $\bar\partial$-Laplacian.
For the symmetric power of $T^*_X$, we can define a raising operator analogous to $L$, which we denote $\mathcal R_G$ and in particular we obtain the identity
\begin{equation}\nonumber
\{ \Box, \mathcal R_G \}
=  \mathcal R_{\Theta(S^mT^*_X)},
\end{equation}
where $\mathcal R_{\Theta(S^mT^*_X)}$ is the raising operator corresponding to the Chern curvature tensor $\Theta$ of $S^mT^*_X$.
If we apply this identity to $\Sigma$, then we obtain that $\mathcal R_G$ preserves the eigenspaces of the Laplacian (Corollary~\ref{ker}).

\medskip
Throughout this article, for a map $F\colon U\rightarrow W$ where $U$ and $V$ are open subsets in the complex Euclidean spaces, $F_k$ denotes the $k$-th component
of the map $F$. Moreover we will use the multi-index.
For example, for $I=(i_1,\ldots, i_n)$ we denote
$e_1^{i_1}\cdots e_n^{i_n}$ by $e^I$,
$(T_zw)_1^{i_1}\cdots(T_zw)_n^{i_n}$ by $(T_zw)^I$,
$i_1+\cdots+i_n$ by $|I|$ and $i_1!\ldots i_n!$ by $I!$.

\bigskip

{\bf Acknowledgment}
The authors would like to thank Takeo Ohsawa and Masanori Adachi for their helpful comments on the first draft.
The Authors were partially supported by Basic Science Research Program through the National Research Foundation of Korea (NRF) funded by the Ministry of Education (NRF-2019R1F1A1060175).

\section{Preliminaries}
In this section we collect some properties of the Bergman kernel, automorphisms and volume forms of  the unit ball for the future use. For more details see, for example, Chapter 1 in \cite{Zhu}.

Let $\mathbb B^n = \{ z\in \mathbb C^n : |z|^2 <1 \}$ be the unit ball in the complex Euclidean space $\mathbb C^n$.
Let $K \colon \mathbb B^n \times \overline{\mathbb B}^n \rightarrow \mathbb C$ denote its Bergman kernel, i.e.
$$
	K(z,w) =  \frac{1}{(1-z\cdot \overline w)^{n+1}}
$$
where $z\cdot \overline w = \sum_{j=1}^n z_j \overline w_j$.
Let
$$
B(z) = (B_{\alpha\overline\beta}) =  \frac{1}{n+1} \left(
\begin{array}{ccc}
\frac{\partial^2}{\partial \overline z_1\partial z_1} \log K(z,z) & \cdots & \frac{\partial^2}{\partial \overline z_1\partial z_n} \log K(z,z)  \\
\vdots &\ddots &\vdots\\
\frac{\partial^2}{\partial \overline z_n\partial z_1} \log K(z,z) & \cdots & \frac{\partial^2}{\partial \overline z_n\partial z_n}\log K(z,z)
\end{array}\right)
$$
and let $g(z) = \sum B_{\alpha\overline\beta} dz_\alpha \otimes d\overline z_\beta$ be the Bergman metric of $\mathbb B^n$.
For a fixed point $z\in \mathbb B^n$,
let $T_z$ be an automorphism of $\mathbb B^n$ defined by
\begin{equation}\label{Tzw}
	T_z(w) = \frac{ z- P_z(w) - s_z Q_z(w)}{1- w\cdot \overline z},
\end{equation}
where $s_z = \sqrt{ 1-|z|^2}$ with $|z|^2 = z\cdot \overline z$,
$P_z$ is the orthogonal projection from $\mathbb C^n$ onto the one dimensional subspace $[z]$ generated by $z$,
and $Q_z$ is the orthogonal projection from $\mathbb C^n$ onto $\mathbb C^n \backslash [z]$.
One has
$$
P_z(w) = \frac{w \cdot \overline z}{|z|^2 } z  \quad \text{ and }\quad
Q_z(w) = w -  \frac{w \cdot \overline z}{|z|^2 } z.
$$
Moreover $T_z$ is an involution, i.e. $T_z\circ T_z(w) = w$  and it satisfies
\begin{equation}\label{formula1}
1- T_{z}(w) \cdot\overline{z} = \frac{1-|z|^2}{1- w \cdot \overline{z}}
\end{equation}
for any $w\in \mathbb B^n$.
Let $\delta$ be a real-valued function on $\mathbb B^n\times \mathbb B^n$
given by
\begin{equation}
\delta := 1- |T_zw |^2.
\end{equation}
Remark that
\begin{equation}\label{formula7}
1-|T_zw|^2=\frac{(1-|z|^2) (1-|w|^2)}{  |1-z\cdot \overline w|^{2}}
\end{equation}
and
$-\log \delta$ is a plurisubharmonic exhaustion on $\mathbb B^n \times\mathbb B^n$ which is strictly plurisubharmonic off the diagonal
and invariant under the diagonal action of $ \text{Aut}(\Omega)$ (cf. \cite{Hwang_To, Seo}).

\begin{lemma} \label{preliminaries}
\hfill
\begin{enumerate}
\item
$dT_z (0) = -(1-|z|^2)P_z - \sqrt{ 1-|z|^2} Q_z$
\item
$ J_{\mathbb R} T_z(0) = (1-|z|^2)^{n+1}$
\item
$dT_z (z) = -\frac{P_z}{1-|z|^2} - \frac{Q_z}{\sqrt{ 1 - |z|^2}}$
\item
$B(z) = (dT_{z} (z))^{*} (dT_{z} (z)) $
\item
Fix $\gamma\in \text{Aut}(\mathbb B^n)$.
Then $B(z) = \overline{d \gamma(z)}^t B(\gamma z) d\gamma(z)$ where
$$d\gamma = \left(\begin{array}{ccc}
\frac{\partial \gamma_1}{\partial z_1} &\cdots &\frac{\partial \gamma_1}{\partial z_n}\\
\vdots& \ddots& \vdots\\
\frac{\partial \gamma_n}{\partial z_1} &\cdots &\frac{\partial \gamma_n}{\partial z_n}
\end{array}\right).$$
\end{enumerate}
\end{lemma}

Let $A=(A_{jk}) := dT_{z} (z)$ and define
\begin{equation}\label{e}
	e_j := \sum_k A_{jk}dz_k.
\end{equation}
\begin{lemma}
$\{ e_j \}_{j=1}^{n}$ is an orthonormal frame of $T_{\mathbb B^n}^{*}$ with respect to the Bergman metric on $\mathbb{B}^n$.
\end{lemma}
\begin{proof}
Since the matrix representation of the given metric on $T_{\mathbb B^n}^{*}$ is $( B(z)^{-1})^{*}$, by (4) of Lemma~\ref{preliminaries} and the orthonormality of $dz_k$'s with respect to the Euclidean metric on $\mathbb C^n$, we have $\langle A (dz_k) , A (dz_m) \rangle_{g}= \delta_{km}$.
\end{proof}

Let $X_1, \ldots, X_n$ be the frame on $T\mathbb B^n$ dual to
$e_1,\ldots, e_n$,
i.e.
$$
	X_j = \sum_k  A^{kj}\frac{\partial}{\partial z_k}
$$
where $(A^{kj})$ denote the inverse matrix of $(A_{jk})$, i.e. $\sum_j A^{kj} A_{jl} = \delta_{kl}$.
Define
$$
	\Gamma_l^{j\mu} :=  \sum_{k,s}  \overline A^{kj} \frac{\partial A_{ls} }{\partial \overline z_k} A^{s\mu}.
$$
Then we have
\begin{equation}\label{Xe}
\overline X_j e_l = \sum_{k,s} \overline A^{kj}\frac{\partial}{\partial \overline z_k} ( A_{ls} dz_s)
=  \sum_{k,s,\mu} \overline A^{kj}\frac{\partial A_{ls}}{\partial \overline z_k} A^{s\mu} e_\mu
= \sum_\mu \Gamma_{l}^{j\mu} e_\mu.
\end{equation}

Let $\Gamma\subset \text{Aut}(\mathbb B^n)$ be a cocompact lattice and let
$\Sigma := \mathbb B^n/\Gamma$.
Without ambiguity we denote the induced
metric by $g$ on $\Sigma = \mathbb B^n/\Gamma$.
Let $\Omega$ be the quotient of $\mathbb B^n\times \mathbb B^n$ under the diagonal action
of $\Gamma$.
We will use the volume form $dV_\omega$ on $\Omega$ defined by
$\omega \wedge \overline \omega$ where $\omega$ is
given by
\begin{equation}\nonumber
\omega = \left( \frac{\sqrt{-1}}{2} \right)^n K(z,z) dz \wedge d\overline z
+  \left( \frac{\sqrt{-1}}{2} \right)^n \frac{|K(w,z)|^2}{K(z,z)} dw \wedge d\overline w
\end{equation}
with $dz:= dz_1\wedge \cdots \wedge dz_n$ and $dw := dw_1\wedge \cdots \wedge dw_n$.
Note that $\omega$ is an invariant $(n,n)$-form
under the diagonal action of $\text{Aut}(\mathbb B^n)$ on $\mathbb B^n\times \mathbb B^n$.
That is \begin{equation}\nonumber
dV_\omega = \left( \frac{\sqrt{-1}}{2} \right)^{2n} |K(w,z)|^2 dz\wedge d\overline z\wedge dw \wedge d\overline w
\end{equation}
and it is an $\text{Aut}(\mathbb B^n)$-invariant
volume form on $\Omega$.
Now for measurable functions $f,h$ on $\Omega$
we set
\begin{equation}\nonumber
\left< f,h \right>_{\alpha} : = c_{\alpha} \int_{\Omega} f\overline h \delta^{\alpha} \,dV_\omega
\end{equation}
where $c_{\alpha} = \frac{\Gamma(n+\alpha+1)}{n! \Gamma(\alpha+1)}$.
Define a weighted $L^2$-space by setting
\begin{equation}\nonumber
L^2_{\alpha}(\Omega) := \{ f : f \text{ is a measurable function on }
\Omega, \|f \|^2_{\alpha} := \left< f,f\right>_{\alpha} < \infty \}
\end{equation}
and a weighted Bergman space by
$A^2_{\alpha}(\Omega) = L^2_{\alpha}(\Omega) \cap \mathcal O(\Omega).$

\section{Raising operators on symmetric powers of the cotangent bundles}\label{raising operator}
\subsection{Hodge type identities}
Let $X$ be a K\"ahler manifold of dimension $n$ and $g$ be its K\"ahler metric.
Let $S^mT_X^*$ be the $m$-th symmetric power of holomorphic cotangent bundle $T^*_X$ of $X$.
For $u\in S^m T^*_X$ and $v\in S^\ell T^*_X $, we will denote
by $uv$ the symmetric product of $u$ and $v$.
We will denote by $\Lambda^{p,q}T^*_X$ the vector bundle of complex-valued $(p,q)$-forms over $X$.
For any $\tau \colon C^\infty (X, S^mT^*_X) \rightarrow C^\infty (X, S^m T^*_X \otimes \Lambda^{p,q} T^*_X)$, define a map
$$
	\mathcal R_\tau^m \colon
	C^\infty(X, S^m T^*_X )
	\rightarrow
	C^\infty(X, S^{m+p} T^*_X \otimes \Lambda^{0,q} T_{X}^{*})
$$
by
$$
\mathcal R^m_\tau (u) = \sum  \tau_{PQ}(u) e^{P} \otimes \overline e^Q
$$
where $\tau (u) = \sum_{|P| = p, |Q|=q} \tau_{PQ}(u) \otimes e^P \otimes \bar e^Q$ for $u\in C^\infty(\Sigma, S^mT^*_X)$.
Note that $\mathcal R^m_\tau$ does not depend on the choice of frame $e$. Moreover
it is globally well-defined on $X$.
If there is no ambiguity, we will denote $\mathcal R_\tau^m$ by $\mathcal R_\tau$.

\begin{example}\label{G}
For the K\"ahler metric $g$, let
$G \in C^\infty(X, \Lambda^{1,1}T^*_X)$ be its K\"ahler form.
Then we may
consider $G$ as a map from
$C^\infty (X, S^mT^*_X)$ to
$ C^\infty (X, S^mT^*_X \otimes \Lambda^{1,1} T^*_X )$
so that $G(u) = u\otimes G$.
Hence for the local frame $\{e_1,\ldots, e_n\}$
so that $g= \sum_\ell e_\ell \otimes \bar e_\ell$,
the corresponding map $\mathcal R_{G}$ is defined as follows:
 \begin{equation}
\begin{aligned}
 \mathcal R_{G} : C^{\infty} \left( X, S^{m}  T^{*}_{X} \right) &\rightarrow C^{\infty} \left( X, S^{m+1} T^{*}_{X}  \otimes \Lambda^{0,1}  T_{X}^{*}  \right)\\
u= \sum_{J} u_{J} e^{J}  &\mapsto \sum_{J,l} (u_{J} e^{J} e_l ) \otimes \bar  e_l.
\end{aligned}
\end{equation}

For a holomorphic coordinate system $(z_1, \cdots, z_n)$ on $X$,
express
$g= \sum g_{\alpha \beta} dz_{\alpha} \otimes d \bar z_{\beta}$
with $g_{\alpha \beta} = g\left(\frac{\partial} {\partial z_{\alpha}}, \frac{\partial} {\partial \bar z_{\beta} }\right)$. Then for any $u \in C^{\infty}(X, S^{m}T_{X}^{*})$,
$
\mathcal R_G (u)
=  \sum_{\alpha,\beta}g_{\alpha \beta}\, u \,  dz_{\alpha} \otimes d \bar z_{\beta}
$
and hence we have the following lemma.
\begin{lemma}\label{RG_local}
Suppose that $u \in C^{\infty}(X, S^{m}T_{X}^{*})$. Then
$$
\mathcal R_G (u) =  \sum_{\alpha} u  dz_{\alpha} \otimes d\bar z_{\alpha} + O\left(|z|^2\right)
$$
for a normal coordinate system $(z_1, \cdots, z_n)$ on $X$ with respect to $g$.
\end{lemma}
\end{example}

\begin{example}\label{connection}
Let $g^{-m}$ be the metric on $S^m T^*_X$ induced from the K\"ahler metric $g$ on $X$.
Let $$
	D_{m}\colon C^\infty (X, S^mT^*_X)\rightarrow C^\infty (X, S^mT^*_X\otimes \Lambda^{1,0} T^*_X )
$$
be the $(1,0)$ part of the Chern connection on $(S^{m} T_{X}^{*} , g^{-m})$.
Then in a local coordinate system $(z_1,\ldots, z_n)$ on $X$
$$
 \mathcal R_{D_m} : C^{\infty}(X, S^{m} T^{*}_{X} ) \rightarrow C^{\infty} \left(X, S^{m+1} T^{*}_{X}  \right)
$$
is given by
\begin{equation}\label{connection_local}
\mathcal R_{D_m} \bigg(\sum_{\alpha} u_{\alpha} h_{\alpha}\bigg)
= \sum_{k,\alpha} \frac{\partial u_{\alpha}}{\partial z_k}h_{\alpha} dz_k
+ \sum_{\alpha,\mu, \beta} u_{\alpha}\theta_{\alpha \beta} ^{\mu} dz_{\beta}  h_{\mu},
\end{equation}
where  ${\theta}_{\alpha}^{\mu}
= \sum_{\beta}  {\theta}_{\alpha \beta}^{\mu}dz_{\beta}$
is the $(1,0)$ part of the connection one form of $D_m$. \qed
\end{example}

\begin{example}\label{curvature}
The Chern curvature
$$\Theta(S^m T^*_X )\in
C^\infty \left(X, \text{Hom}(S^m T^*_X , S^m T^*_X )
\otimes \Lambda^{1,1} T^*_X \right) $$
with respect to $g^{-m}$ is given by
\begin{equation}\label{curvature}
\Theta(S^m T^*_X ) s_1s_2\cdots s_m = \sum_{1\leq j\leq m} s_1 s_2 \cdots \Theta(T^*_X)\cdot s_j\cdots s_m
\end{equation}
with $s_j\in C^\infty(X, T^*_X)$ for each $j=1,\ldots, m$.
If we denote $\Theta(S^m T^*_X ) =
\sum_{} R_{\alpha k \bar l}^\beta h_\alpha^*\otimes h_\beta \otimes dz_k\wedge d\bar z_l$ with a holomorphic local frame $h_\alpha$ of $S^m T^*_X $, we have
\begin{equation}\label{curvature2}
\mathcal R_{\Theta(S^m T^*_X )}u
 = \sum_{\alpha, \beta, k,l} u_\alpha R_{\alpha k \bar l}^\beta h_\beta dz_k\otimes d\bar z_l
\end{equation}
for $u= \sum_\alpha u_\alpha h_\alpha$.
\qed
\end{example}

For the Hermitian metric $g^{-m}$ on $S^mT_X^*$
and  the K\"ahler metric $g$ on $X$, consider the Laplacian
$$
\Box^{s}_{m}  : C^{\infty}(X, S^{m} T_{X}^{*}  \otimes \Lambda^{0,s} T_{X}^{*} ) \rightarrow
C^{\infty}(X, S^{m} T^{*}_{X}  \otimes \Lambda^{0,s} T^{*}_{X} )
$$
defined by
$$
\Box^s_{m} =  \bar \partial \bar \partial^{*}+ {\bar \partial}^{*} \bar \partial.
$$
If there is no confusion, we will omit subscript $m$ and superscript $s$ for brevity.

\medskip

For given $m \in \mathbb{N}$, $l \in \mathbb{N} \cup \{0 \}$,
and $r \in \mathbb{Z}$, let $\mathcal A_{r,l}^{m}$ be the set of operators which map
\begin{equation}
S^{m}_{k} T^{*}_{X}  \rightarrow S^{m+l}_{k+r}  T^{*}_X ,
\end{equation}
with $S^{m}_{k} T_X^{*} := S^{m} T_X^{*} \otimes \Lambda^{k} T^{*} X$.
For fixed $\ell$, $r$, consider a family of operators $A=\{ A_m \}$ where $A_m \in \mathcal A_{r,\ell}^m$,
i.e. $A \subset \bigoplus_{m} \mathcal A_{r,\ell}^{m}$.
Let $A \subset \bigoplus_{m} \mathcal A^{m}_{a,1} $ and $B \subset \bigoplus_{m} \mathcal A^{m}_{b,0}$ be such families of operators. Define
\begin{equation}
\{A, B\}_{m} := A_{m} \circ B_{m} - (-1)^{ab} B_{m+1} \circ A_{m}
\end{equation}
and denote the collection of $\{ A, B \}_{m}$ by $\{A, B \}$. Note that $\{A, B \} \in \bigoplus_{m} \mathcal A^{m}_{a+b ,1} $.
Also, one defines
\begin{equation}
\{B, A\}_{m} = B_{m+1} \circ A_{m} - (-1)^{ba} A_{m} \circ B_{m}
\end{equation}
and it is easy to verify that $\{A, B \}_{m} = (-1)^{ab+1} \{B, A \}_{m}$.

\begin{lemma}\label{brac}
$$
\{\Box, \mathcal R_G \} = \Box \circ \mathcal R_G - \mathcal R_G \circ \Box = \{ \bar \partial, \{ \bar \partial^{*}, \mathcal R_G \} \} + \{ \bar \partial^{*} , \{\mathcal R_{G}, \bar \partial \} \}
$$
\end{lemma}
\begin{proof}
Note that $\mathcal R_G\in \bigoplus_m \mathcal A^m_{1,1}$ and $\Box \in \bigoplus_m \mathcal A^m_{0,0}$.
Since
\begin{equation*}
\begin{aligned}
\{\bar \partial, \{ \bar \partial^{*} ,\mathcal R_G \}\}
&= \bar \partial_{m+1} \{ \bar \partial^{*} ,\mathcal R_G \} - (-1)^{1 \cdot 0} \{ \bar \partial^{*} ,\mathcal R_G \} \bar \partial_{m} \\
&= \bar \partial_{m+1}
\left(
	\bar \partial^{*}_{m+1} \mathcal R_G - (-1)^{1 \cdot (-1)} \mathcal R_G \bar \partial^{*}_{m}
\right)
- \left(
	\bar \partial^{*}_{m+1} \mathcal R_G - (-1)^{1 \cdot (-1)} \mathcal R_G \bar \partial_{m}^{*}
\right)
\bar \partial_{m} \\
&= \bar \partial_{m+1} \bar \partial^{*}_{m+1}\mathcal R_G + \bar \partial_{m+1} \mathcal R_G\bar \partial^{*}_{m} - \bar \partial^{*}_{m+1} \mathcal R_G \bar \partial_{m} -\mathcal R_G \bar \partial^{*}_{m} \bar \partial_{m}
\end{aligned}
\end{equation*}
and
\begin{equation*}
\begin{aligned}
\{ \bar \partial^{*}, \{\mathcal R_G, \bar \partial \} \}
&= \bar \partial^{*}_{m+1} \{ \mathcal R_G, \bar \partial \}
- (-1)^{-1 \cdot 2} \{\mathcal R_G, \bar \partial \} \bar \partial^{*}_{m} \\
&= \bar \partial_{m+1}^{*}
\left(\mathcal R_G \bar \partial_{m} - (-1)^{1\cdot 1} \bar \partial_{m+1} \mathcal R_G\right)
-  \left( \mathcal R_G \bar \partial_{m} - (-1)^{1 \cdot 1} \bar \partial_{m+1} \mathcal R_G\right) \bar \partial^{*}_{m} \\
&= \bar \partial_{m+1}^{*} \mathcal R_G \bar \partial_{m} + \bar \partial^{*}_{m+1} \bar \partial_{m+1} \mathcal R_G - \mathcal R_G \bar \partial_{m} \bar \partial^{*}_{m} - \bar \partial_{m+1} \mathcal R_G \bar \partial^{*}_{m},
\end{aligned}
\end{equation*}
it follows that
\begin{equation*}
\{ \bar \partial, \{ \bar \partial^{*}, \mathcal R_G \} \}
+ \{ \bar \partial^{*}, \{\mathcal R_G, \bar \partial \} \}
= \Box_{m+1} \mathcal R_G
 - \mathcal R_G \Box_{m}= \{ \Box, \mathcal R_G\}.
\end{equation*}
\end{proof}

\begin{lemma} \label{Laplacian_RG}
Let $(X,g)$ be a K\"ahler manifold. Let $S^mT^*_X$ and $T^*_X$ equip the Hermitian metrics $g^{-m}$ and $g$.
Then for any $u\in C^\infty(X, S^mT_X^*)$,
\begin{equation}\nonumber
\{ \Box, \mathcal R_G \} u
=  \mathcal R_{\Theta(S^mT^*_X)}u.
\end{equation}
\end{lemma}
\begin{proof}
Fix a point $p\in X$.
Let $(z_1, \cdots, z_n)$ be a normal coordinate system on $X$ at $p$
with respect to $g$.
For any $u = \sum_\alpha u_\alpha h_\alpha \in C^\infty(X, S^{m} T_X^{*})$ where $\{h_\alpha\}$ is a holomorphic normal frame of $S^mT_X^*$ with respect to the chosen normal coordinate system $(z_1, \cdots, z_n)$ on $X$, we have
\begin{equation}\label{connection_local2}
\begin{aligned}
\{{\bar \partial}^{*}, \mathcal R_G \} u
&= {\bar \partial}^{*} \mathcal R_G (u) - (-1)^{1 \cdot 1} \mathcal R_G {\bar \partial}^{*} (u)
= {\bar \partial}^{*} \mathcal R_G(u) \\
&= {\bar \partial}^{*} \left(  \sum_{\alpha, \beta,\gamma}   g_{\alpha\bar\beta} u_\gamma h_\gamma dz_\alpha  \otimes d\bar z_\beta \right)
=  \sum_{\alpha, \beta, \gamma}h_{\gamma} dz_\alpha  \otimes \bar \partial^{*}_{g} (u_{\gamma}g_{\alpha\bar\beta} d\bar z_\beta)+O(|z|)  \\
&=  \sum_{\alpha, \beta,\gamma } h_{\gamma}  dz_\alpha \otimes
\left(- \sum_{k}  \frac{\partial}{\partial \overline z_k} \lrcorner \frac{\partial (u_{\gamma}g_{\alpha\bar\beta})}{\partial z_k}d \bar z_\beta \right) + O(|z|)\\
&= -  \sum_{\alpha, \beta,\gamma} \left( h_{\gamma}  dz_\alpha \otimes \sum_{k}g_{\alpha\bar\beta}\frac{\partial u_{\gamma}}{\partial z_k} \delta^{k}_{\beta} \right) +O(|z|) \\
&= -\sum_{\alpha, \beta, \gamma} g_{\alpha\bar\beta}\frac{\partial u_{\gamma}} {\partial z_\beta} h_{\gamma}  dz_\alpha + O(|z|).
\end{aligned}
\end{equation}
By evaluating \eqref{connection_local}, \eqref{connection_local2} at the point $p$ we obtain
\begin{equation}\label{brac1}
\{{\bar \partial}^{*}, \mathcal R_{G}\} u = - \mathcal R_{D_m} u.
\end{equation}
Therefore by \eqref{connection_local} and the relation $\Theta(S^mT_X^*) = \bar\partial \theta$,
we have
\begin{equation}\label{curvature1}
\begin{aligned}
{\bar \partial} \mathcal R_{D_m} u
&= \sum_{k,\alpha, \mu} \frac{\partial^2 u_{\alpha}}{\partial z_k \partial \bar z_{\mu}}  dz_k  h_{\alpha} \otimes d \bar z_{\mu}
+ \sum_{\alpha, \beta, \mu, \zeta} u_{\alpha} \frac{\partial \theta_{\alpha \beta}^{\mu}} {\partial \bar z_{\zeta}}  dz_{\beta} h_{\mu} \otimes d \bar z_{\zeta} + O(|z|)\\
&=\sum_{k,\alpha, \mu} \frac{\partial^2 u_{\alpha}}{\partial z_k \partial \bar z_{\mu}}  dz_k  h_{\alpha} \otimes d \bar z_{\mu}
- \sum_{\alpha, \beta, \mu, \zeta} u_\alpha R_{\alpha\beta\bar \zeta}^\mu dz_\beta h_\mu \otimes d\bar z_\zeta
+ O(|z|)
\end{aligned}
\end{equation}
with $\Theta (S^{m} T_X^{*} )
= \sum_{\alpha,\mu}\Omega_{\alpha}^{\mu}h_\alpha^* \otimes h_\mu
= \sum_{\alpha,\mu,\beta, k} R^{\mu}_{\alpha \beta \bar k} h_\alpha^* \otimes h_\mu \otimes d z_{\beta} \wedge d \bar z_{k}$.

On the other hand by Lemma \ref{RG_local}, we have
$$
\mathcal R_G \bar \partial^{*} \bar \partial u
=   - \sum_{\mu,\alpha,\beta,\gamma} \frac{\partial^2 u_{\alpha}}{\partial z_{\mu} \partial \bar z_{\mu}} h_{\alpha}  g_{\gamma\bar\beta}dz_\gamma \otimes d \bar z_\beta+ O(|z|)
$$
and
\begin{equation}\nonumber
\begin{aligned}
\bar \partial^{*} \mathcal R_G \bar \partial u
&= \bar \partial^{*} \mathcal R_G \left( \sum_{\mu,\alpha} \frac{\partial  u_{\alpha}}{\partial \bar z_{\mu}} h_{\alpha}\otimes d\bar z_{\mu} \right)
=  \bar \partial^{*}\left( \sum_{\mu, \alpha,\beta,\gamma} g_{\gamma\bar\beta}h_{\alpha}  dz_{\gamma}\otimes \frac{\partial u_{\alpha}}{\partial \bar z_{\mu}}  d \bar z_{\mu} \wedge d\bar z_\beta \right)\\
&=\sum_{\mu, \alpha,\beta,\gamma} g_{\gamma\bar\beta}h_{\alpha}  dz_\gamma \otimes \sum_{k}
\left( -\frac{\partial^2 u_{\alpha}}{\partial \bar z_{\mu} \partial z_{k}} \delta_{k}^{\mu} d\bar z_{\beta} + \frac{\partial^2 u_{\alpha}}{\partial z_k \partial \bar z_{\mu}} \delta_{\beta}^{k} d \bar z_{\mu}\right) +O(|z|) \\
&=
\sum_{\mu, \alpha,\beta,\gamma}g_{\gamma\bar\beta} h_{\alpha}
\left( -\frac{\partial^2 u_{\alpha}}{\partial \bar z_{\mu} \partial z_{\mu}}  dz_\gamma \otimes d\bar z_{\beta}
+ \frac{\partial^2 u_{\alpha}}{\partial z_\beta \partial \bar z_{\mu}}   dz_\gamma \otimes d \bar z_{\mu}\right) + O(|z|).
\end{aligned}
\end{equation}
Hence we have
\begin{equation}\label{brac2}
\begin{aligned}
\{ \bar \partial^{*}, \mathcal R_G \} \bar \partial u
&= \bar \partial^{*} \mathcal R_G  \bar \partial u + \mathcal R_G  \bar \partial^{*} \bar \partial u\\
&=  -  2\sum_{\mu,\alpha,\beta,\gamma} \frac{\partial^2 u_{\alpha}}{\partial z_{\mu} \partial \bar z_{\mu}} h_{\alpha}  g_{\gamma\bar\beta}dz_\gamma \otimes d \bar z_\beta
+  \sum_{\mu, \alpha,\beta,\gamma}g_{\gamma\bar\beta} h_{\alpha} \frac{\partial^2 u_{\alpha}}{\partial z_\beta \partial \bar z_{\mu}}   dz_\gamma \otimes d \bar z_{\mu}
 + O(|z|).
\end{aligned}
\end{equation}
By \eqref{brac2}, \eqref{brac1} and \eqref{curvature1} we obtain
\begin{equation}\nonumber
\begin{aligned}
\{\bar \partial, \{ {\bar \partial}^{*} , \mathcal R_G \} \}u
&= \bar \partial \{ {\bar \partial}^{*}, \mathcal R_G \}u - \{ {\bar \partial}^{*} , \mathcal R_G \} \bar \partial u
= -  \bar \partial \mathcal R_{D_m} u - \{ {\bar \partial}^{*} , \mathcal R_G \} \bar \partial u \\
=& -2\sum_{k,\alpha, \mu} \frac{\partial^2 u_{\alpha}}{\partial z_k \partial \bar z_{\mu}}  dz_k  h_{\alpha} \otimes d \bar z_{\mu}
+ \sum_{\alpha, \beta, \mu, \zeta} u_\alpha R_{\alpha\beta\zeta}^\mu dz_\beta h_\mu
\otimes d\bar z_\zeta \\
&  +  2 \sum_{\mu, l,\alpha} \frac{\partial^2 u_{\alpha}}{\partial  z_{\mu} \partial \bar z_{\mu}} h_{\alpha}  dz_l \otimes d\bar z_{l}
\end{aligned}
\end{equation}
at the point $p$. Similarly at $p$ we obtain
\begin{equation}\nonumber
\begin{aligned}
\{\bar \partial^{*}, \{ \mathcal R_G , \bar \partial \}  \}u
&= \bar\partial^{*} \{ \mathcal R_G, \bar \partial \} u- \{ \mathcal R_G, \bar \partial \} \bar \partial^{*}u
= \bar \partial^{*} \mathcal R_G \bar \partial u+ \bar \partial^{*} \bar \partial \mathcal R_Gu \\
&= - 2 \sum_{\mu, l,\alpha} h_{\alpha}  dz_l\otimes \frac{\partial^2 u_{\alpha}}{\partial \bar z_{\mu} \partial z_{\mu}} d\bar z_{l}
+ 2 \sum_{\mu,l,\alpha} h_{\alpha}  dz_l\otimes \frac{\partial^2 u_{\alpha}}{\partial z_l \partial \bar z_{\mu}} d \bar z_{\mu}
\end{aligned}
\end{equation}
and hence by Lemma \ref{brac} we obtain
\begin{equation}\nonumber
\{ \Box, \mathcal R_G \} u
= \{ \bar \partial, \{ \bar \partial^{*}, \mathcal R_G \} \}u + \{ \bar \partial^{*}, \{\mathcal R_G, \bar \partial \} \}u
= \sum_{\alpha, \beta, \mu, \zeta} u_\alpha R_{\alpha\beta\bar \zeta}^\mu dz_\beta h_\mu
\otimes d\bar z_\zeta
\end{equation}
at $p$ which implies the lemma.
\end{proof}

\subsection{Hodge type identities over the complex hyperbolic space form}
Let $\Sigma = \mathbb B^n /\Gamma$ and $g$ be the metric induced from the Bergman metric on $\mathbb B^n$. Then $S^{m}T_{\Sigma}^{*}$ has the induced metric $g^{-m}$.
For any measurable section $\phi$ of $S^{m}T_{\Sigma}^{*} \otimes \Lambda^{p,q} (\Sigma)$, we define an $L^2$-norm by
\begin{equation*}\label{defnorm}
\begin{aligned}
\| \phi \|^2= \int_{\Sigma} \langle \phi, \phi \rangle  dV_{\Sigma}
\end{aligned}
\end{equation*}
where $\langle~ ,~  \rangle$ and $dV_\Sigma$ are induced by $g$ on $\Sigma$. Hence if $\phi= \sum_{I} f_I e^{I}$ is a measurable section of $S^{m}T_{\Sigma}^{*}$,
\begin{equation*}
\| \phi \|^2 = \int_{\Sigma} \sum_{|I|=m} \frac{I!}{m!} |f_{I}|^2 dV_{\Sigma}.
\end{equation*}

Suppose $m$ is sufficiently large and let  $v $ be an element of $ C^{\infty} (\Sigma, S^{N+m} T^{*}_{\Sigma} \otimes \Lambda^{0,1}T^*_\Sigma)$ satisfying $\bar \partial v=0$.
For the Green operator $G^{1}$ of $S^{N+m} T^{*}_{\Sigma} $-valued $(0,1)$-forms,
we have
$$
\Box^{1}\circ G^{1}v = v
$$
if $\ker\Box^{1}=0$.
In particular, $u:=\bar \partial^{*} G^{1}v$ is the $L^2$ minimal solution for $\bar \partial$-equation $\bar \partial u =v$ when $\bar \partial v=0$ and $\ker \Box^{1}=0$.

\begin{proposition}
$\mathcal R_G$
is a linear injective map and for any $u\in C^\infty(\Sigma, S^mT^*_\Sigma)$
\begin{equation}\label{norm}
\|\mathcal R_G (u) \|^2 = \frac{m+n}{m+1}  \| u \|^2
\end{equation}
and
\begin{equation}
\{ \Box, \mathcal R_G \} (u)
 = 2m \mathcal R_G(u).
\end{equation}
\end{proposition}

\begin{proof}
It is clear that the map $\mathcal R_G$ is an injective linear map. Let $u = \sum_{|I|=m} f_I e^{I}$. Then the equation \eqref{norm} can be induced by
the following:
\begin{equation}\nonumber
\begin{aligned}
\| \mathcal R_G(u) \|^2
&=\int_{\Sigma} \sum_{\ell,s} \langle u  e_\ell \otimes \bar e_\ell , u  e_s \otimes \bar e_s \rangle_{g^{-m-1} \otimes g} dV
=  \int_{\Sigma} \sum_{\ell,s} \langle u  e_\ell , u  e_s \rangle_{g^{-m-1}} \langle \bar e_\ell, \bar e_s \rangle_{g} dV \\
&=  \int_{\Sigma}  \sum_{\ell}  \langle u  e_\ell , u  e_\ell \rangle_{g^{-m-1}}  dV
= \int_{\Sigma} \sum_{\ell} \sum_{|I|=m} \frac{i_1! \cdots (i_{\ell}+1)! \cdots i_n!}{(m+1)!}  |f_{I}|^2 dV \\
&=  \frac{m+n}{m+1}\| u \|^2. \\
\end{aligned}
\end{equation}

By \eqref{curvature} we have
\begin{equation}\label{curvature_ball}
\begin{aligned}
\mathcal R_{\Theta(S^{m}T_{\Sigma}^{*}) }\left(\sum_{|I|=m} f_{I} e^{I}\right)
&= \sum_{|I|=m} f_{I} \sum_{j} i_{j} e_1^{i_1}  \cdots  e_{j}^{i_j -1}  \cdots e_n^{i_n}  \mathcal R_{\Theta (T_{\Sigma}^{*})} (e_j).
\end{aligned}
\end{equation}
Since the Chern curvature tensor of $T_\Sigma$ with respect to $g$ is given by
$$
{ \Theta (T_{\Sigma})  }^{a}_{b}
= -   \left( e_{a} \wedge \bar e_{b} + \delta_{ab} \sum_{r} e_r \wedge \bar e_r\right),
$$
one has
\begin{equation}\nonumber
{\Theta (T_{\Sigma}^{*})}_{b}^{a}
=  e_{b} \wedge \bar e_{a} + \delta_{ab} \sum_{r} e_r \wedge \bar e_r.
\end{equation}
This implies
\begin{equation}\nonumber
\begin{aligned}
\Theta (T_{\Sigma}^{*}) (e_j) &= \sum_{a} e_{a}\otimes \Theta (T^{*}_{\Sigma})^{a}_{j}
= \sum_{a}  e_{a} \otimes \left(e_j \wedge \bar e_{a} + \delta_{ja} \sum_{r} e_r \wedge \bar e_r\right) \\
&=  \sum_{a}  e_{a} \otimes e_j\wedge \bar e_a +  \sum_{r}  e_j\otimes e_r \wedge \bar e_r
\end{aligned}
\end{equation}
and hence by \eqref{curvature_ball} we obtain
\begin{equation}\nonumber
\begin{aligned}
\mathcal R_{\Theta(S^{m}T_{\Sigma}^{*})}\left(\sum_{|I|=m} f_{I}e^I\right)
&=  \sum_{|I|=m} f_{I}\left( \sum_{j} i_{j} e_1^{i_1}  \cdots  e_{j}^{i_j -1}  \cdots e_n^{i_n} \left(
2\sum_{a} e_j  e_a\otimes \bar e_a  \right) \right)\\
&=2 \sum_{|I|=m} \sum_{j} i_j f_{I} e^{I} \sum_{a} e_a \otimes \bar e_a
= 2 m\mathcal R_G u.
\end{aligned}
\end{equation}
By Lemma \ref{Laplacian_RG}, we obtain the lemma.
\end{proof}

\begin{corollary}\label{ker}
\begin{equation}\nonumber
\mathcal R_G^m ( \ker ( \Box^{0}_{m} - \lambda I) ) \subset \ker \left(\Box^{1}_{m} - \left(\lambda + 2m \right) I \right).
\end{equation}
\end{corollary}
\begin{proof}
Let $\lambda$ be an eigenvalue of $u$ for $\Box^{0}$. Then
$$
\{ \Box, \mathcal R_G^m \} u
= \Box \mathcal R_G^m (u) - \mathcal R_G^m (\lambda u)
=2m\mathcal R_G^m(u).
$$
Therefore,
$\Box \mathcal R_G^m (u) = (2m + \lambda) \mathcal R_G^m(u)$ follows.
\end{proof}

\section{Construction of $L^2$ holomorphic functions on $\Omega$}
\subsection{Precomputations}

\begin{lemma} \label{BC}
Let
$$
	\frac{\partial {(T_z w)}_{k}}{\partial \bar z_j } \bigg |_{w=T_z t} := \sum_I B_I(z) t^I
\quad\text{ and } \quad
\frac{\partial (T_z w)_k}{\partial z_j} \bigg|_{w=T_z t} := \sum_I C_I(z) t^I
$$
be the expansion in $t$ variable.
Then $B_I=0$ if $|I| \geq 3$ and $C_I=0$ if $|I| \geq 2$.
\end{lemma}

\begin{proof}
By differentiating ${(T_z w)}_{k}$ with respect to $\frac{\partial}{\partial \bar z_j}$ and $\frac{\partial}{\partial z_j}$ respectively, we obtain

\begin{equation*}
\begin{aligned}
\frac{\partial {(T_z w)}_{k} }{\partial \bar z_j}
&=   \frac{- \frac{\partial }{\partial \bar z_j} \left( P_z (w) + s_z Q_z (w) \right)_k}{1-w \cdot \bar z} +  \frac{\left( z- P_z (w) - s_z Q_z (w) \right)_{k}w_j}{(1-w \cdot \bar z)^2}\\
&=\frac{- \frac{\partial }{\partial \bar z_j} \left( P_z (w) + s_z Q_z (w) \right)_k}{1-w \cdot \bar z} +  \frac{\left(T_z w\right)_{k}w_j}{1-w \cdot \bar z}\\
\end{aligned}
\end{equation*}
and
$$
\frac{\partial {(T_z w)}_{k} }{\partial z_j}
= \frac{\frac{\partial }{\partial z_j} \left( z- P_z (w) - s_z Q_z (w) \right)_k }{1-w \cdot \bar z}
= \frac{\delta_{jk}}{1-w \cdot\bar z} -  \frac{ \frac{\partial }{\partial z_j} (P_z (w) + s_z Q_z (w))_k}{1-w \cdot\bar z}.
$$
Note that $ -\frac{\partial}{\partial z_j}
\left( P_z (w) + s_z Q_z(w) \right)$ and $-\frac{\partial}{\partial \bar z_j}
\left( P_z (w) + s_z Q_z (w) \right)$ are linear operators
and holomorphic in the $w$ variable.
We will denote these by $L_{z} (w)$ and $L_{\bar z} (w)$, respectively.

By \eqref{formula1} we obtain

\begin{equation*}
\begin{aligned}
\frac{\partial {(T_z w)}_{k} }{\partial \bar z_j} \bigg |_{w=T_z t}
&=\frac{ L_{\bar z} (T_z t)_k }{1- T_z t \cdot \bar z}
+ \frac{t_k (T_z t)_j}{1- T_z t \cdot\bar z} \\
&= L_{\bar z}\left(\frac{z - P_z(t) - s_z Q_z (t)}{1- t \cdot\bar z}   \right)_k \frac{1-t \cdot\bar z} {1-|z|^2} + \frac{ t_k\left( z- P_z (t) - s_z Q_z (t)\right)_{j} }{1-|z|^2} \\
&=  \frac{L_{\bar z}\left(z - P_z(t) - s_z Q_z (t)\right)_k} {1-|z|^2}
+ \frac{ t_k\left( z- P_z (t) - s_z Q_z (t)\right)_{j} }{1-|z|^2} \\
\end{aligned}
\end{equation*}
and

\begin{equation}\label{rmk$L_{z}$}
\begin{aligned}
\frac{\partial {(T_z w)}_{k} }{\partial z_j} \bigg |_{w=T_z t}
&= \frac{\delta_{jk}}{1- T_z t\cdot \bar z}
+ \frac{L_{z} (T_z t)_k }{1- T_z t \cdot \bar z}  \\
&= \delta_{jk} \frac{1- t \cdot\bar z}{1-|z|^2}
+ L_{z}\left( \frac{z - P_z(t) - s_z Q_z (t)}{1- t\cdot \bar z} \right)_k \frac{1-t \cdot\bar z} {1-|z|^2} \\
&= \delta_{jk} \frac{1- t \cdot\bar z}{1-|z|^2}
+ \frac{L_{z}(z - P_z(t) - s_z Q_z (t))_k}{1-|z|^2}.
\end{aligned}
\end{equation}
Since $z-P_z (t)-s_z Q_z(t)$ is holomorphic and linear in $t$ variable,
$\frac{\partial(T_z w)_{k}}{\partial \bar z_j}\big|_{w=T_z t}$ is a polynomial of degree two and
$\frac{\partial (T_z w)_{k}}{\partial z_j} \big|_{w=T_z t}$ is a polynomial of degree one in $t$ variable.
\end{proof}

\begin{lemma}\label{formula3}
\begin{equation*}
\sum_{k=1}^{n} \left( \frac{\partial (T_z w)_{k}}{\partial \bar z_j }
\bigg |_{w=T_z t} \right) \overline{z_k}
= \left( -1 + \frac{s_z}{1-|z|^2} \right) t_j  + \frac{(1-s_z) t \cdot\bar z}{|z|^2 (1-|z|^2)} z_j - \frac{s_z t_j t\cdot \bar z }{1-|z|^2}- \frac{(1-s_z) (t\cdot\bar z)^2}{|z|^2 (1-|z|^2)}z_j.
\end{equation*}
\end{lemma}
\begin{proof}
By differentiating \eqref{formula1} with respect to $-\frac{\partial}{ \partial \overline z_j}$, we have
\begin{equation}\label{formula2}
\sum_{k=1}^{n} \left( \frac{\partial (T_z w)_{k}}{\partial \bar z_j } \right) \overline z_k = - (T_{z} w)_{j} -\frac{\partial}{\partial \bar z_j} \left(\frac{1-|z|^2}{1-w\cdot \overline z} \right).
\end{equation}
Since
\begin{equation}\nonumber
\frac{\partial}{\partial \bar z_j} \left( \frac{1-|z|^2}{1-w \cdot\bar z} \right) = \frac{-z_j}{1-w \cdot \bar{z}} +  \frac{w_j(1-|z|^2)}{(1-w\cdot \bar{z})^2}
\end{equation}
and
\begin{equation}\label{eq1}
{\left( T_{z} t \right)}_j = \frac{z_j - \frac{t\cdot\bar{z}}{|z|^2} z_j - s_z \left( t_j - \frac{t \cdot\overline{z}}{|z|^2} z_j  \right)} {1-t\cdot\bar{z}},
\end{equation}
by \eqref{formula1} we obtain
\begin{equation}\nonumber
\frac{\partial}{\partial \bar z_j}
\left( \frac{1-|z|^2}{1-w \cdot\bar z} \right) \bigg |_{w=T_{z} t}
= - \frac{1-t\cdot \bar z}{1-|z|^2}
\left( \frac{t\cdot\bar{z}}{|z|^2}  z_j + s_z \left( t_j - \frac{t \cdot\overline{z}}{|z|^2} z_j  \right) \right)
\end{equation}
and hence we obtain the lemma.
\end{proof}

\begin{lemma}
\begin{equation}\label{formula9}
\begin{aligned}
&\sum_{k=1}^{n} \left(
\frac{\partial {(T_z w)}_{k} } {\partial \bar z_j }\bigg |_{w=T_z t} \overline {t_k}
+
{\frac{\partial {(\overline{T_z w})}_{k} }{\partial  \overline z_j} } \bigg |_{w=T_z t} {t}_{k}
\right) \\
&= -\sum_l \frac{z_j \bar z_l (s_z -1)}{|z|^2(1-|z|^2)} t_l + \frac{s_z}{1-|z|^2} t_j
- \sum_{l} \frac{s_z}{1-|z|^2} t_l t_j \bar t_l +
\sum_{l,m} \frac{z_j \bar z_m(s_z -1)}{|z|^2 (1-|z|^2)} t_l \bar t_l t_m.
\end{aligned}
\end{equation}
\end{lemma}
\begin{proof}
By differentiating \eqref{formula7} with respect to $-\frac{\partial}{\partial\overline z_j}$, we obtain
\begin{equation}\label{formula8}
\sum_{k=1}^{n}
\left(
\frac{\partial {(T_z w)}_{k} } {\partial \bar z_j }\bigg |_{w=T_z t} \overline {t _k} + \frac{\partial {(\overline{T_z w})}_{k} }{\partial \overline z_j} \bigg |_{w=T_z t}  {t}_{k}
\right)
= -
	\frac{ \partial } { \partial  \bar z_j } \left( \frac{ (1-|z|^2)(1-|w|^2)}{|1- w \cdot\bar z|^{2} }
\right)
\bigg |_{w=T_z t}.
\end{equation}
Since
\begin{equation}\nonumber
\frac{\partial} {\partial \bar z_j} \left( (1-|z|^2)(1-|w|^2)  \right)
= -(1-|w|^2) z_j
\end{equation}
and
\begin{equation}\nonumber
\begin{aligned}
\frac{\partial}{\partial \bar z_j} \frac{1}{|1- w \cdot \bar z|^2}
 = \frac{w_j (1- z \cdot \bar w) }{ |1 - w\cdot \bar z|^4},
\end{aligned}
\end{equation}
one obtains
\begin{equation}\nonumber
\frac{\partial}{\partial \bar z_j} \frac{(1-|z|^2)(1-|w|^2)}{ |1- w \cdot \bar z|^{2}}
=
- \frac{(1-|w|^2) z_j}{|1- w \cdot\bar z|^2} + (1-|z|^2)(1-|w|^2) \frac{w_j (1- z \cdot\bar w)}{|1 -w\cdot \bar z|^4}.
\end{equation}
Since one obtains by \eqref{formula1} and \eqref{formula7}
\begin{equation}\nonumber
-\frac{(1-|w|^2) z_j}{|1- w \cdot\bar z|^2} \bigg |_{w=T_z t}  = - z_j \frac{1-|t|^2}{1-|z|^2}
\end{equation}
and
\begin{equation}\nonumber
(1-|z|^2)(1-|w|^2) \frac{w_j (1- z \cdot \bar w)}{|1 -w \cdot \bar z|^4} \bigg |_{w=T_z t} = \frac{1-|t|^2}{1-|z|^2} \left( z_j - \frac{t \cdot \bar z}{|z|^2} z_j - s_z \left( t_j - \frac{t \cdot \bar z}{|z|^2} z_j \right) \right),
\end{equation}
by \eqref{eq1} and \eqref{formula8}
\begin{equation}\nonumber
\begin{aligned}
&\sum_{k=1}^{n} \left(
\frac{\partial {(T_z w)}_{k} } {\partial \bar z_j }\bigg |_{w=T_z t} \overline {t_k} + \frac{\partial {(\overline{T_z w})}_{k} }{\partial \overline z_j} \bigg |_{w=T_z t}  {t}_{k}
\right)
= \frac{1-|t|^2}{1-|z|^2} \left(\frac{t \cdot\bar z}{|z|^2} z_j + s_z \left( t_j - \frac{t \cdot \bar z}{|z|^2} z_j \right) \right)
\end{aligned}
\end{equation}
and hence the lemma is proved.
\end{proof}

\begin{lemma}
\begin{equation}\label{taylor2}
\sum_{k=1}^n \frac{\partial (T_z w)_k}{\partial z_j} \bigg |_{w=T_z t}   \bar z_k =  \frac{\bar z_j}{1-|z|^2} - \frac{\bar z_j}{1-|z|^2} t \cdot \bar z
\end{equation}
\end{lemma}
\begin{proof}
Since
\begin{equation}\label{eq1}
\sum_{k=1}^n \frac{\partial (T_z w)_{k}}{\partial z_j} \bar z_k
= - \frac{\partial}{\partial z_j} \left( \frac{1-|z|^2}{1-w\cdot \bar z}   \right)
=  \frac{\bar z_j}{1-w \cdot \bar z}
\end{equation}
by differentiating the formula \eqref{formula1} with respect to $\frac{\partial}{\partial z_j}$,
we obtain the lemma by substituting $w=T_zt$ in \eqref{eq1} and
by applying \eqref{formula1}.
\end{proof}

By Lemma~\ref{BC}, we may express $
\frac{\partial {(T_z w)}_{k}}{\partial \bar z_j } \bigg |_{w=T_z t}$ and
$\frac{\partial (T_z w)_k}{\partial z_j} \bigg|_{w=T_z t}$ as the following:

\begin{equation}\label{BC1}
\begin{aligned}
\frac{\partial {(T_z w)}_{k}}{\partial \bar z_j } \bigg |_{w=T_z t}
&= \sum_{l=1}^{n} B^{jk}_{l}(z) t_l +  \sum_{ l > m  }^{n}   B^{jk}_{lm} (z) t_l t_m + \sum_{l=1}^{n} B^{jk}_{ll}(z) t_l t_l\\
\frac{\partial (T_z w)_k}{\partial z_j} \bigg|_{w=T_z t}
&= C^{jk} (z) + \sum_{l} C^{jk}_{l}(z) t_l
\end{aligned}
\end{equation}

\begin{lemma}\label{BC2}
In the expression \eqref{BC1}, we have
\begin{equation*}
\begin{aligned}
&B^{jk}_l = 0 \quad \text{ if } k, l, j \text{ are not identical, }\\
& B^{jk}_{l} + \bar C^{jl}_{k} = 0, \\
& \bar C^{jl}_{k} = 0 \quad \text{if $l \not = k$},
\end{aligned}
\end{equation*}
and
\begin{equation*}
\begin{aligned}
  B^{jk}_{k \alpha } &= - \frac{s_z}{1-|z|^2} \delta_{\alpha j} + \frac{z_j (s_z -1)}{|z|^2 (1-|z|^2)} \bar z_{\alpha} \quad\quad\text{ if } k\geq \alpha,\\
B^{jk}_{k \alpha } &= - \frac{s_z}{1-|z|^2} \delta_{\alpha j} + \frac{z_j (s_z -1)}{|z|^2 (1-|z|^2)} \bar z_{\alpha} \quad\quad\text{ if } k\leq \alpha.
\end{aligned}
\end{equation*}
\end{lemma}

\begin{proof}
By substituting \eqref{BC1} in the equation of Lemma~\ref{formula3}
we see that $B_l^{jk}=0$ if $l$, $j$, $k$ are not identical.
By substituting \eqref{BC1} in the equation \eqref{formula9} we obtain
\begin{equation}\nonumber
\begin{aligned}
& \sum_{k=1}^{n} \left( \overline{ C^{jk}(z)}  t_k +    \sum_{l=1}^{n}
\left( B^{jk}_{l}+ \overline{C^{jl}_{k}}\right) t_l \bar t_k + \sum_{ l > m}^{n}  B^{jk}_{lm}(z)  t_l t_m\overline t_k + \sum_{l=1}^{n} B^{jk}_{ll}(z) t_l t_l \overline t_k   \right)\\
&= -  \sum_l \frac{z_j (s_z -1)}{|z|^2(1-|z|^2)}\bar z_l t_l +\frac{s_z}{1-|z|^2} t_j
-\sum_{l} \frac{s_z}{1-|z|^2} t_l t_j \bar t_l +
\sum_{l,m} \frac{z_j (s_z -1)}{|z|^2 (1-|z|^2)}\bar z_m t_l \bar t_l t_m.
\end{aligned}
\end{equation}
By comparing the coefficients of $t$ variable, we obtain the lemma.
\end{proof}

\begin{lemma}
\begin{enumerate}[label=\rm(\arabic*)]

\item
$ B^{jk}_{l} =\sum_{\alpha} \frac{\partial A_{k \alpha}}{\partial \bar z_j} A^{\alpha l}=0 \quad \text{ if } l \not = k$;
\item
$C^{jk}_{l} = \sum_{\alpha} \frac{\partial A_{k \alpha}}{\partial z_j} A^{\alpha l} - \sum_{\alpha} \frac{\partial^2 (T_z w)_{k}}{\partial w_j \partial w_{\alpha}} \bigg |_{z=w} A^{\alpha l}$
\end{enumerate}
\end{lemma}
\begin{proof}
Note that
\begin{equation}\label{formula5}
\begin{aligned}
B^{jk}_{l} (z) = \frac{\partial}{\partial t_{l}}\bigg |_{t=0} \left( \frac{\partial {(T_z w)}_k}{\partial \bar{z_j} } \bigg|_{w=T_z t} \right)
= \sum_{\alpha} \frac{\partial^2 {(T_z w)}_k }{ \partial w_\alpha \partial \bar z_j} \bigg |_{w=z} \frac{\partial {(T_z t)}_{\alpha}}{\partial t_l} \bigg |_{t=0}
= \sum_{\alpha} \frac{\partial A_{k \alpha}}{\partial \bar z_j} A^{\alpha l}.
\end{aligned}
\end{equation}
The second equality of \eqref{formula5} follows by
\begin{equation*}
\frac{\partial}{\partial \bar z_j }\left( \frac{\partial (T_z w)_k}{\partial w_{\alpha}} \bigg |_{z=w} \right) = \frac{\partial (T_z w)_k}{\partial w_{\alpha}\partial \bar z_j} \bigg|_{z=w},
\end{equation*}
by the chain rule and the fact that $\frac{\partial (T_z w)_k}{\partial w_{\alpha}}$ is holomorphic in $w$ variable.
Hence, (1) is followed by Lemma 3.6. To prove (2) note that
\begin{equation}\nonumber
\begin{aligned}
C^{jk}_{l} &= \frac{\partial}{\partial t_l} \bigg |_{t=0} \left(  \frac{\partial (T_z w)_{k}} {\partial z_j} \bigg |_{w=T_z t} \right)
= \sum_{\alpha} \frac{\partial^2 (T_z w)_{k}} {\partial w_{\alpha} \partial z_{j} } \bigg |_{w=z} \frac{\partial (T_z w)_{\alpha}}{\partial t_l} \bigg |_{t=0}
\end{aligned}
\end{equation}
But
\begin{equation*}
\begin{aligned}
\frac{\partial}{\partial z_j} \left ( \frac{ \partial (T_z w)_{k}} {\partial w_{\alpha}} \bigg |_{w=z} \right ) &= \sum_{\alpha} \bigg (\frac{\partial^2 (T_z w)_{k}} {\partial z_j \partial w_{\alpha}} \bigg |_{w=z} + \frac{\partial (T_z w)_{k}}{\partial w_j \partial w_{\alpha}} \bigg |_{w=z} \bigg )
\end{aligned}
\end{equation*}
Therefore (2) follows.
\end{proof}

In particular we have
\begin{equation}\label{forcor1}
\begin{aligned}
\frac{\partial {(T_z w)}_{k}}{\partial \bar z_j } \bigg |_{w=T_z t} &=  B^{jk}_{k}(z) t_k
+ \sum_{m=1}^{ k-1}  B^{jk}_{km}(z) t_k t_m +  \sum_{l=k+1}^{ n } B^{jk}_{lk}(z)t_l t_k +  B^{jk}_{kk}(z) t_k t_k.
\end{aligned}
\end{equation}

\begin{corollary}\label{z bar}
$$
\frac{\partial ( T_zw)_k}{\partial \overline z_j}\Big|_{w=z}=0$$
\end{corollary}
\begin{proof}
Since $T_{z}(0) = z$, it is followed by \eqref{forcor1}.
\end{proof}

\begin{lemma}\label{ABsum}
$$
	\sum_{j} \overline A^{j \mu} B^{jk}_{k \alpha } = \delta_{ \mu \alpha}
$$
\end{lemma}
\begin{proof}
Since one has
\begin{equation}\nonumber
 A^{j \mu} = -  \frac{s^{2}_{z}\bar z_{\mu} z_{j}}{ |z|^2}- s_{z} \left(\delta_{j \mu} - \frac{ \bar z_{\mu}  z_{j}}{|z|^2}\right),
\end{equation}
by straightforward calculation using Lemma~\ref{BC2} we obtain the lemma.
\end{proof}

\subsection{Necessary condition to be holomorphic functions}
Let $f$ be a holomorphic function on $\Omega$. Then
we may consider $f$ as a $\Gamma$-invariant holomorphic function on $\mathbb B^n\times \mathbb B^n$.
By putting $t=T_zw$, there exists $\tilde f$ so that
\begin{equation}\nonumber
f(z,w) = f(z,T_zt)= \tilde f(z,t).
\end{equation}
Note that $\tilde f(z,t)$ is holomorphic in $t$ but not in $z$.
Express
$$
	\tilde f(z,t) :=\sum_{|I| = 0}^\infty f_{I}(z) t^I
\quad\text{ with }\quad
	f_{I} (z)= \frac{1}{I!} \frac{ \partial^{|I|}\tilde f}{\partial t^I}  (z,0).
$$
Note that $$
f(z,w)=	\sum_{|I|= 0}^\infty   f_{I}(z) (T_zw)^I.
$$
\begin{proposition}\label{necessary condition}
Suppose that $f$ is a holomorphic function on $\Omega$.
Then
$$
	\overline X_\mu f_{I}
+ \sum_k  i_k f_{I} \Gamma^{\mu k}_k
+  (|I|-1)  f_{i_1 \ldots i_{\mu}-1 \ldots i_n}
=0
$$
for each $I = (i_1,\ldots, i_n)$.
\end{proposition}
\begin{proof}
Since $f(z,w)$ is holomorphic on $\Omega$,
we have
 \begin{equation}\label{necessary to be holomorphic}
\begin{aligned}
0&= \frac{\partial}{\partial \overline z_j} f(z,w) = \frac{\partial}{\partial \overline z_j} f(z,w(z,t))
=\frac{\partial}{\partial \overline z_j} \tilde f(z,t(z,w))\\
&=\frac{\partial \tilde f}{\partial \overline z_j}(z,t) + \sum_{k}
\frac{\partial \tilde f}{\partial t_k}(z,t) \frac{\partial ( T_zw)_k}{\partial \overline z_j}\bigg|_{w=T_zt}.
\end{aligned}
\end{equation}
By \eqref{necessary to be holomorphic}, we obtain
\begin{equation}\label{recursion}
\begin{aligned}
0&= \frac{1}{I !} \frac{\partial^{|I|}}{\partial t^I}\bigg|_{t=0}
\left( \frac{\partial \tilde f}{\partial \overline z_j}(z,t)
+\sum_{k} \frac{\partial \tilde f}{\partial t_k}(z,t) \frac{\partial ( T_zw)_k}{\partial \overline z_j}\bigg|_{w=T_zt}\right)\\
&= \frac{\partial f_{I}}{\partial \overline z_j}(z) + \frac{1}{I!} \sum_{k}
\frac{\partial^{|I|}}{\partial t^I}\bigg|_{t=0}
\left(  \frac{\partial \tilde f}{\partial t_k}(z,t) \frac{\partial ( T_zw)_k}{\partial \overline z_j}\bigg|_{w=T_zt}\right).
\end{aligned}
\end{equation}
Since
$$
	\frac{\partial \tilde f}{\partial t_k} = \sum_{|I|=1, i_k \not = 0}^{\infty} i_k f_{i_1 \cdots i_n} t_1^{i_1} \cdots t_k^{i_k -1} \cdots t_n^{i_n},
$$
by Lemma~\ref{BC} we have
\begin{equation}\nonumber
\begin{aligned}
\frac{\partial \tilde{f}}{\partial t_k} \frac{\partial (T_z w)_k} {\partial \bar z_j} \bigg |_{w=T_z t}
& =\sum_{ |I| =1, i_k \not = 0}^{\infty} i_k f_{i_1 \cdots i_n} t_1^{i_1} \cdots t_k^{i_k-1} \cdots t_n^{i_n} \left( B_{k}^{jk} t_k \right) \\
&+  \sum_{|I|=1, i_k \not =0 }^{\infty} i_k f_{i_1\cdots i_n} t_1^{i_1} \cdots t_k^{i_k-1} \cdots t_n^{i_n} \sum_{ m=1}^{ k-1}  B^{jk}_{km}(z) t_k t_m \\
&+ \sum_{|I|=1, i_k \not =0 }^{\infty} i_k f_{i_1 \cdots i_n} t_1^{i_1} \cdots t_k^{i_k-1} \cdots t_n^{i_n}
\left(\sum_{l=k+1}^{n} B^{jk}_{lk}(z)t_l t_k +  B^{jk}_{kk}(z) t_k t_k \right).
\end{aligned}
\end{equation}

Therefore we have
\begin{equation}\nonumber
\begin{aligned}
\frac{1}{I !} \frac{\partial^{|I|}}{\partial t^I} \bigg |_{t=0}
\left (
\frac{\partial \tilde{I}}{\partial t_k} \frac{\partial (T_z w)_k} {\partial \bar z_j} \bigg |_{w=T_z t}
\right )
=&  i_k f_{i_1 \ldots i_n} B^{jk}_{k} + i_k  \sum_{m=1}^{k-1} f_{i_1 \ldots i_{m}-1 \ldots i_n} B_{km}^{jk} \\
& +  i_k  \sum_{l=k+1}^{n } f_{i_1\ldots i_l - 1 \ldots i_n}  B_{lk}^{jk} +
 (i_k - 1) f_{i_{1} \ldots i_k -1   \ldots, i_n} B_{kk}^{jk}
\end{aligned}
\end{equation}
and hence by \eqref{recursion}, Lemma~\ref{ABsum} implies that
\begin{equation}\nonumber
\begin{aligned}
0&= \sum_{j}  \overline A^{j\mu}
\left(
\frac{\partial f_{i_1\ldots i_n}}{\partial \overline z_j}(z) + \frac{1}{I!} \sum_{k}
\frac{\partial^m}{\partial t^I}\bigg|_{t=0}
\left(  \frac{\partial \tilde I}{\partial t_k}(z,t) \frac{\partial ( T_zw)_k}{\partial \overline z_j}\bigg|_{w=T_zt}\right)
\right)\\
&= \overline X_\mu f_{i_1 \ldots i_n}
+ \sum_k \bigg( i_k f_{i_1 \ldots i_n} \sum_j \overline A^{j\mu} B^{jk}_{k}
+ i_k  \sum_{m=1}^{k-1} f_{i_1 \ldots i_{m}-1 \ldots i_n} \sum_j  \overline A^{j\mu} B_{km}^{jk} \\
 &\quad\quad\quad\quad\quad
+  i_k  \sum_{l=k+1}^{n } f_{i_1\ldots i_l - 1 \ldots i_n} \sum_j  \overline A^{j\mu} B_{lk}^{jk} +
 (i_k - 1) f_{i_{1} \ldots i_k -1   \ldots i_n}  \sum_j  \overline A^{j\mu} B_{kk}^{jk} \bigg)\\
&= \overline X_\mu f_{i_1 \ldots i_n}
+\sum_k   i_k f_{i_1 \ldots i_n} \Gamma^{\mu k}_k
+ \sum_k \bigg( {i_k}  \sum_{m=1}^{k-1} f_{i_1 \ldots i_{m}-1 \ldots i_n} \delta_{\mu m}\\
&\quad\quad \quad\quad \quad\quad \quad\quad \quad\quad
+  {i_k  }   \sum_{l=k+1}^{n } f_{i_1\ldots i_l - 1 \ldots i_n} \delta_{\mu l} +
 (i_k - 1)  f_{i_{1} \ldots i_k -1   \ldots i_n}\delta_{\mu k} \bigg) \\
&= \overline X_\mu f_{i_1 \ldots i_n}
+ \sum_k  i_k f_{i_1 \ldots i_n} \Gamma^{\mu k}_k
+  (|I|-1) f_{i_1 \ldots i_{\mu}-1 \ldots i_n}.
\end{aligned}
\end{equation}
\end{proof}

For a $\Gamma$-invariant holomorphic function
$
\sum_{|I|= 0}^\infty   f_{I}(z) (T_zw)^I
$
 on $\Omega$, define

\begin{displaymath}
\varphi_{I} := \left\{ \begin{array}{ll}
f_{I}(z) e^I & \text{when $|I|=m,~m \geq 1$} \\
0 & \text{when $|I|=0$} \end{array}, \right.
\end{displaymath}
$$\varphi_k := \sum_{|I|=k}\varphi_I,$$
and
\begin{equation}\label{associated differential}
\varphi (z) := \sum_{|I| =0 }^\infty\varphi_{I} \in \bigoplus_{m=0}^\infty
C^\infty(\Sigma, S^mT^*_\Sigma).
\end{equation}
We will call $\varphi$ the {\it associated differential} of $f$.
For fixed $m$ and $I=(i_1, \cdots, i_n)$ with
$|I| =m$, we have
\begin{equation}\nonumber
\begin{aligned}
\bar \partial \varphi_{i_1 \ldots i_n}
&=  \sum_{\mu} \left( \overline X_{\mu}  f_{i_1 \ldots i_n} +\sum_k i_k f_{i_1 \ldots i_n}
\Gamma_{k}^{\mu k}\right) e_1^{i_1}  \cdots  e_n^{i_n}\otimes \overline e_{\mu}\\
&=-\sum_{\mu} (|I|-1) f_{i_1 \ldots i_{\mu}-1 \ldots i_n} e_1^{i_1}  \cdots  e_n^{i_n} \otimes \overline e_{\mu}\\
&= - (|I|-1)
\sum_\mu \varphi_{i_1\ldots i_\mu-1 \ldots i_n} e_\mu \otimes \overline e_\mu,
\end{aligned}
\end{equation}
which implies
\begin{equation}\label{necessary}
\bar \partial \varphi_{k} = -(k-1)  \mathcal R_G\left( \varphi_{k-1}\right)
\end{equation}
since for fixed $\mu$, we have $\sum_{|I|=k} \varphi_{i_1\ldots i_{\mu}-1\ldots i_n}
= \sum_{|I|=k-1} \varphi_{i_1\ldots i_n}$.

\bigskip

\begin{proof}[Proof of Theorem~\ref{SD}]
Let $f\in \mathcal O(\Omega)$ which vanishes up to $k$-th order on $D$. Then
$\varphi_m \equiv 0$ for any $m\leq k$
but $\varphi_{k+1} \not\equiv 0$ and hence we have $\varphi_{k+1} \in H^0(\Sigma, S^{k+1}T^*_\Sigma)$ by the equations \eqref{necessary}.
Define $\Psi (f) = \varphi_{k+1}$. Then the proof is completed.
\end{proof}

\begin{proof}[Proof of Corollary~\ref{SD_cor}]
For a bounded domain $D$ in $\mathbb C^n$ if $\Gamma$ is a discrete subgroup of $\text{Aut}(\mathbb B^n)$,
it is known by Poincar\'e that $\sum_{\gamma\in\Gamma} |\mathcal J_{\mathbb C} \gamma(z)|^2$
locally uniformly converges to a smooth function on $D$ where $\mathcal J_{\mathbb C}\gamma$
denotes the determinant of complex Jacobian matrix of $\gamma$.
By straightforward calculation, for any $\gamma\in \text{Aut}(\mathbb B^n)$, we obtain
$$
(1-|\gamma^{-1}(0)|^2) |\gamma(z)-\gamma(w)|^2
\leq |\mathcal J_{\mathbb C}\gamma(z)|^{\frac{2}{n+1}}
|\mathcal J_{\mathbb C}\gamma(w)|^{\frac{2}{n+1}}|z-w|^2.
$$
Hence
$$
	\sum_{\gamma\in \Gamma}(1-|\gamma^{-1}(0)|^2)^{N/2} \sum_{j=1}^n (\gamma_j(z)-\gamma_j(w))^{N}
$$
is a $\Gamma$-invariant holomorphic function on $\mathbb B^n\times\mathbb B^n$ for any $N\geq n+1$
with respect to the diagonal action.
By Theorem~\ref{SD}, there exists a symmetric differential of degree $N$ for any $N\geq n+2$.
\end{proof}
\begin{corollary}
Let $\Gamma\subset \text{Aut}(\mathbb B^n)$ be a torsion-free Kottwitz lattice,
and $\Omega=\mathbb B^n\times \mathbb B^n/\Gamma$
the corresponding holomorphic $\mathbb B^n$-fiber bundle.
Assume $n+1$ is prime. Then there is no holomorphic function which
vanishes on $D$ up to order $k$ for $1<k\leq n-1$ on $\Omega$.
\end{corollary}

\begin{proof}
Suppose there is a holomorphic function which vanishes on $D$
of order $k$ with $0<k \leq n-1$ on $\Omega$.
Then by Theorem \ref{SD} there exists a symmetric differential
$\psi\in H^0(\Sigma, S^kT^*_\Sigma)$. However by Theorem 1.11 in \cite{Klingler},
there is no symmetric differential of degree $1, \ldots, n-1$.
\end{proof}

\subsection{Proof of Theorem \ref{main} }
In this section we assume that $\Sigma$ is compact.
First we establish a vanishing theorem for $H^{0,1}_{\bar \partial} (\Sigma, S^{m}T^{*}_{\Sigma})$.

\begin{definition}
A line bundle $L$ on a K\"{a}hler manifold $X$ is said to be positive if there exists a Hermitian metric $h$ on $L$ with the Chern curvature form $\sqrt{-1} \Theta(L)$ is a positive $(1,1)$ form.
\end{definition}

\begin{theorem}[Kodaira-Nakano vanishing theorem]
If $(E,h)$ is a positive line bundle on a compact K\"ahler manifold $(X,\omega)$, then

$$H^{p,q}_{\bar \partial}(X,E) =0$$ for $p+q\geq n+1$.
\end{theorem}

\begin{proposition}\label{existence_holo}
Let $\Sigma=\mathbb B^n /\Gamma$ be a compact complex hyperbolic space form.
For $m \geq n+2$,
$$
	H^{0,1}_{\bar \partial} (\Sigma, S^m T^*_{\Sigma})=0.
$$
\end{proposition}

\begin{proof} The following proof is influenced by the argument in \cite{Wong}.
Let $\mathbb{P}T_{\Sigma}$ be the projectivization of the holomorphic tangent bundle $(T_{\Sigma}, g) \rightarrow \Sigma$. Then one has the associated line bundle $(\mathcal{O}_{\mathbb{P}T_{\Sigma}}(-1), \hat g) \rightarrow \mathbb{P} T_{\Sigma}$, which is obtained by the fiberwise Hopf blow-up process.  Now we impose a K\"{a}hler form $\omega_{\mathbb{P}T_{\Sigma}}$ on $\mathbb{P} T_{\Sigma}$ by $-c_1(\mathcal{O}_{\mathbb{P} T_{\Sigma}} (-1), \hat g)$, where $c_1$ denotes the first Chern class.

Let $K_{\mathbb{P}T_{\Sigma}}$ be the canonical line bundle over $\mathbb{P}T_{\Sigma}$. Then a direct calculation yields that $K^{-1}_{\mathbb{P}T_{\Sigma}} \otimes \mathcal{O}_{\mathbb{P}T_{\Sigma}}(m)$ is positive if $m\geq n+2$ (see \cite{Mok, Wong}). Therefore the Kodaira-Nakano vanishing theorem guarantees that $H^{1}(\mathbb{P}T_{\Sigma}, \mathcal{O}_{\mathbb{P} T_{\Sigma}}(m)) \cong  H^{n,1}_{\bar \partial}(\mathbb{P}T_{\Sigma}, K^{-1}_{\mathbb{P}T_{\Sigma}} \otimes \mathcal{O}_{\mathbb{P}T_{\Sigma}}(m))=0$. Since $H^{1} (\Sigma, S^{m}T_{\Sigma}^{*}) \cong H^{1}(\mathbb{P}T_{\Sigma}, \mathcal{O}_{\mathbb{P}T_{\Sigma}}(m))$, if $m \geq n+2$ then $H^{0,1}_{\bar \partial} (\Sigma, S^{m} T^{*}_{\Sigma})$ vanishes.
\end{proof}

Let $N \geq n+2$. For a given $\psi \in H^{0}(\Sigma, S^{N}T_{\Sigma}^{*})$,
define $\varphi (z) := \sum_{k =0 }^\infty\varphi_{k} \in \bigoplus_{k=0}^\infty
C^\infty(\Sigma, S^mT^*_\Sigma)$ by
\begin{equation}\label{system}
\left\{ \begin{array}{ll}
\varphi_{k}=0 & \text{if $k<N$}, \\
\varphi_{N}=\psi ,&
\end{array} \right.
\end{equation}
and for $m \geq 1$,  $\varphi_{N+m}$ is the solution of
the following $\overline \partial$-equation:
\begin{equation}\label{system2}
\bar \partial \varphi_{N+m} = - (N+m -1) \mathcal R_G\left( \varphi_{N+m-1}  \right)
\end{equation}
By the following lemma and Proposition \ref{existence_holo} the $L^2$ minimal solution of \eqref{system2}
\begin{equation}\label{Step2-2}
\bar \partial^{*} G^{1} \left( -  (N+m -1)\mathcal R_G(\varphi_{N+m-1}  ) \right)
\end{equation}
exists.
\begin{lemma}\label{solsuff}
$\mathcal R_G(\varphi_{N+m-1})$ is $\bar\partial$-closed.
\end{lemma}
\begin{proof}
Let $\{dz^{I} \}$ be a local holomorphic frame on $S^{N+m-2}T_{\Sigma}^{*}$ where $dz^{I}$
denotes $dz_1^{i_1}\ldots dz_n^{i_n}$.
We may express $\varphi_{N+m-1}$ by
$$
\varphi_{N+m-1} = \sum_{|I|= N+m-1} \varphi_I dz^I.
$$
By \eqref{system2} we have
$$
\bar \partial \varphi_{N+m-1}
= \sum_{|I|=N+m-1} \sum_{j}
\overline X_j \varphi_I dz^I \otimes \bar e_j
= -(N+m-1) \sum_l\varphi_{N+m-2}  e_l \otimes \bar e_l
$$
implying
\begin{equation}\label{check1}
\overline X_j \varphi_I dz^I
= - (N+m-1)  \varphi_{N+m-2}  e_j
\end{equation}
for any $j$. Therefore we have
\begin{equation}\nonumber
\begin{aligned}
\bar \partial \left(
\sum_{l}\varphi_{N+m-1}  e_l\otimes \bar e_l
\right)
=  \sum_{l,j,I}  \overline X_j \varphi_I dz^I  e_l \otimes \bar e_l \wedge \bar e_j
 + \sum_{m,j,l} \varphi_{N+m-1}  dz_m \otimes ( \bar \partial (A_{lm} \overline{A_{lj}}) \wedge d \bar z_j ).
\end{aligned}
\end{equation}
Since by \eqref{check1}
\begin{equation}\nonumber
\begin{aligned}
\sum_{l,j,I}  \overline X_j \varphi_I dz^I  e_l \otimes \bar e_l \wedge \bar e_j
&= - \sum_{j,l}(N+m-1)\varphi_{N+m-2}  e_j  e_l \otimes \bar e_l \wedge \bar e_j =0,
\end{aligned}
\end{equation}

and by Lemma~\ref{BC} and Lemma \ref{BC2} we have
\begin{equation}\nonumber
\begin{aligned}
& \sum_{j,l} \bar \partial (A_{lm} \overline{A_{lj}}) \wedge d \bar z_j
= \sum_{j, l,k} \left( \frac{ \partial A_{lm}}{\partial \bar z_k} \bar A_{lj}
+ A_{lm} \frac{\partial \bar A_{lj}}{\partial \bar z_k}\right) d \bar z_{k} \wedge d \bar z_j\\
&= \sum_{j,l,k} \left( \sum_s B_s^{kl} A_{sm} \overline A_{lj}
+ A_{lm}\left( \overline{\sum_s A_{sj}C^{kl}_s
+  \frac{\partial^2 (T_zw)_l}{\partial w_k\partial w_j} \bigg|_{z=w}} \right)
\right) d\bar z_k \wedge d \bar z_j =0.
\end{aligned}
\end{equation}

\end{proof}

\begin{lemma}\label{norm of varphi}
Let
$\{ \varphi_k \}_{k=0}^{\infty} \in \bigoplus_{k=0}^{\infty} C^{\infty} (\Sigma, S^k T_{\Sigma}^{*})
$ be the sequence defined by \eqref{system} and \eqref{system2}. Then
$$
\| \varphi_{N+m} \|^2 =  \bigg( \prod_{j=1}^{m} \left(1+ \frac{n-1}{N+j} \right) \bigg) \left( \frac{(2N-1)!}{ \{ (N-1)!  \}^2}
\frac{ \{ (N+m-1)! \}^2} { (2N+m-1)!}\frac{1}{m! } \right) \| \psi \|^2
$$
for any $m \geq 1$.
\end{lemma}
\begin{proof}

First we will show that $\varphi_{N+m}$ is an eigenvector of $\Box^0$.
Let $E_{N,m}$ be its eigenvalue.
Since $\varphi_{N} = \psi$ is a holomorphic section,
$\varphi_N$ is an eigenvector of $\Box^0$ with eigenvalue $E_{N,0}=0$.
Suppose that $\varphi_{N+m}$ is an eigenvector of $\square^{0}$ for some $m\geq 0$.
By Corollary~\ref{ker} we have
\begin{equation}\label{G}
 \mathcal R_G \left( \varphi_{N+m}\right)=
\Box^{1} G^{1} \left( \mathcal R_G\left( \varphi_{N+m}\right)\right)
= G^{1} \left(
\left( {2 (N+m)} + E_{N,m}
\right)
\mathcal R_G (\varphi_{N+m})\right),
\end{equation}
and by $ \Box^0 \bar\partial^* = \bar\partial^* \Box^1$ and $\Box^1 G^1 = G^1 \Box^1$ we have
\begin{equation}\nonumber
\begin{aligned}
\Box^{0} \varphi_{N+m+1}
&= \Box^{0} \bar \partial^{*} G^{1}
\left(
- (N+m) \mathcal R_G( \varphi_{N+m})
\right)\\
&= \left( 2 (N+m)  + E_{N,m}
\right)
 \bar \partial^{*} G^{1} \left(
 - (N+m)\mathcal R_G( \varphi_{N+m})
\right) \\
&= \left( 2 (N+m) + E_{N,m} \right) \varphi_{N+m+1}
\end{aligned}
\end{equation}
which implies that $\varphi_{N+m+1}$ is an eigenvector of $\Box^{0}$
and $E_{N+m+1} =  2 (N+m)  + E_{N,m}$.
Moreover we have
\begin{equation}\nonumber
E_{N,m}
= 2 \left(0 +N + (N+1) + \cdots + (N+m-1) \right)
= m(2N+m-1).
\end{equation}

By \eqref{system2}, \eqref{Step2-2}, \eqref{G} and \eqref{norm} we have
\begin{equation}\nonumber
\begin{aligned}
\| \varphi_{N+m} \|^2
&= (N+m-1)^2 \langle \langle \bar \partial^{*} G^{1}
 \mathcal R_G( \varphi_{N+m-1}), \bar \partial^{*} G^{1}
 \mathcal R_G( \varphi_{N+m-1}) \rangle \rangle \\
&= (N+m-1)^2   \langle \langle G^{1} \mathcal R_G( \varphi_{N+m-1}),
\mathcal R_G( \varphi_{N+m-1}) \rangle \rangle \\
&=  \frac{(N+m-1)^2}{2(N+m-1) + E_{N,m-1} } \| \mathcal R_G( \varphi_{N+m-1})\|^{2} \\
&= \left( 1+ \frac{n-1}{N+m} \right)\frac{ (N+m-1)^2}{ E_{N,m} } \| \varphi_{N+m-1} \|^2.
\end{aligned}
\end{equation}
Continuing this process we obtain
\begin{equation}\nonumber
\begin{aligned}
\| \varphi_{N+m} \|^2
&= \left(
\prod_{j=1}^{m} \left(1+ \frac{n-1}{N+j} \right) \frac{(N+m-j)^2}{E_{N,m-j+1}}
\right) \| \psi \|^2 \\
&=  \left(
\prod_{j=1}^{m} \left(1+ \frac{n-1}{N+j} \right) \frac{(N+m-j)^2 }{ (m-j+1)(2N+m-j)  }
\right) \| \psi \|^2 \\
&= \bigg( \prod_{j=1}^{m} \left(1+ \frac{n-1}{N+j} \right) \bigg) \frac{(2N-1)! \{ (N+m-1)!\}^2 }{\{(N-1)!\}^2 (2N+m-1)!  }\frac{1}{m!} \| \psi \|^2.
\end{aligned}
\end{equation}
\end{proof}

Let us express
$$
\varphi_{\ell} = \sum_{|I| = \ell} f_{I} e^I
$$ and define
a formal sum $f$ on $\Omega$ by
\begin{equation}\label{f}
f(z,w) =  \sum_{|I|= 0}^\infty  f_{I} (z) (T_z w)^I.
\end{equation}

\begin{lemma}
$f(z,w)$ is $\Gamma$-invariant.
\end{lemma}
\begin{proof}
Fix $\gamma \in \text{Aut}(\mathbb B^n)$.
Since \begin{equation}\label{transformation rule}
T_{\gamma z} \gamma w = U_z T_zw
\end{equation}
for some unitary matrix $U_z$ depending only on $z$,
we have $dT_{\gamma z} |_{\gamma w} d\gamma |_w = U_z dT_z|_w$
and in particular $$U_z =  dT_{\gamma z} |_{\gamma z} d\gamma |_z dT_z |_0.$$

Since $\phi\in \bigoplus_{k=0}^\infty H^0 \left(\Sigma, S^{k}T^*_\Sigma \right)$,
we have $\gamma^* \phi_k = \phi_k$. Note that
\begin{equation}\nonumber
\gamma^* e_j = \sum_k  A_{jk}(\gamma z) d\gamma_k
= \sum_{k,m,l} A_{jk}(\gamma z) \frac{\partial \gamma_k }{\partial z_l} A^{lm} e_m
\end{equation}
where $(A^{lm})$ denotes the inverse matrix of $A$, i.e.
 \begin{equation*}
\gamma^* e = A(\gamma z) d\gamma(z) A^{-1} e = U_z e.
\end{equation*}
This implies
\begin{equation}\label{invariance}
\begin{aligned}
 \sum_{|I| = 0}^\infty f_{I}(z) e^I
= \sum_{|I| = 0}^\infty f_{I}(\gamma z) (\gamma^*e)^I
= \sum_{|I| = 0}^\infty f_{I}(\gamma z) (U_ze)^I
\end{aligned}
\end{equation}
and hence by \eqref{transformation rule} and \eqref{invariance} we have
\begin{equation}\nonumber
f(\gamma z, \gamma w)
=  \sum_{|I| = 0}^\infty f_{I}(\gamma z) (T_{\gamma z} \gamma w )^I
=  \sum_{|I| = 0}^\infty f_{I}(\gamma z) (U_z T_zw)^I
=\sum_{|I| = 0}^\infty f_{I}(\gamma z) ( T_zw)^I
= f(z,w).
\end{equation}
\end{proof}

\begin{lemma}\label{norm of f}
Let $f$ be a formal sum given in \eqref{f}. Then
\begin{equation}\nonumber
\| f\|^2_{\alpha} =\frac{ \pi^{n}}{n!} \sum_{|I|=0}^\infty ||\varphi_{|I|}||^2
\frac{ |I|! \Gamma(n+\alpha+1) }{\Gamma(n+|I|+\alpha+1)}.
\end{equation}
\end{lemma}

\begin{proof}
Let $\tilde\Sigma$ denote the fundamental domain of $\Sigma$ in $\mathbb B^n$
and $\tilde\Omega$ denote the corresponding domain of $\Omega$
in $\tilde \Sigma \times \mathbb{B}^n \subset \Omega$.
Then
\begin{equation}\label{norm of f_1}
\begin{aligned}
\|f \|^2_{\alpha}
&= c_{\alpha} \int_{\Omega}  \bigg|\sum_{|I|= 0}^\infty  f_{I} (z) (T_z w)^I\bigg|^2 ( 1-|T_z w |^2)^{\alpha}  |K(z,w)|^2 d\lambda_{z}d \lambda_{w}.
\end{aligned}
\end{equation}
Since $t=T_zw$,
$ J_{\mathbb R} T_z(0)
 = (1-|z|^2)^{n+1}$, $d\lambda_w
= | J_{\mathbb C}T_zt|^2 d\lambda_t$
and
\begin{equation}\nonumber
K(z,w) = K(T_z 0, T_z t) = \frac{ K(0,t) }{J_{\mathbb C}T_z(0) \overline{ J_{\mathbb C} T_z(t)}}
= \frac{ 1}{J_{\mathbb C}T_z(0) \overline{ J_{\mathbb C} T_z(t)}},
\end{equation}
by \eqref{norm of f_1} we obtain
\begin{equation}\label{norm of f_2}
\begin{aligned}
||f||^2_{\alpha}
& = c_{\alpha} \int_{\Sigma} \frac{1}{(1-|z|^2)^{n+1}} d \lambda_{z} \int_{\mathbb{B}^n}
\bigg|\sum_{|I|= 0}^\infty    f_{I} (z) t^I\bigg|^2 ( 1-|t|^2 )^{\alpha}   d\lambda_{t} \\
&= c_{\alpha}  \int_{\Sigma} \frac{1}{(1-|z|^2)^{n+1}} d\lambda_{z} \int_{\mathbb{B}^n}
\sum_{|I|= 0}^\infty    \left| f_{I} (z) t^I \right|^2 \big( 1-|t|^2  \big)^{\alpha} d \lambda_{t}.
\end{aligned}
\end{equation}
The second equality in \eqref{norm of f_2} can be induced by the orthogonality of polynomials
with respect to the inner product $\int_{\mathbb{B}^n} f \bar g (1-|t|^2)^{\alpha} d \lambda_{t}$ (see \cite{Zhu}).
Since we have
\begin{equation}\nonumber
dV_{\Sigma} = \text{det} ( B(z) ) =  K(z,z) d\lambda_{z}
\end{equation}
and
\begin{equation}\nonumber
\begin{aligned}
\| \varphi_{l} \|^2
= \sum_{|I|=l}\int_{\Sigma} \frac{I!}{l!}\frac{\left|f_{I} \right|^2} {(1-|z|^2)^{n+1}} d\lambda_{z},
\end{aligned}
\end{equation}
by \eqref{norm of f_2} one has
\begin{equation}\nonumber
\begin{aligned}
||f||^2_{\alpha} &=
\sum_{|I|= 0}^\infty c_{\alpha} \int_{\Sigma} \frac{|f_I(z)|^2}{(1-|z|^2)^{n+1}}
d\lambda_z \int_{\mathbb{B}^n} |t^I|^2 (1-|t|^2)^{\alpha} d \lambda_t \\
&= \frac{\pi^{n}}{n!} \sum_{|I|=0}^\infty ||\varphi_{|I|}||^2 \frac{|I|!
\Gamma(n+\alpha+1) }{\Gamma(n+|I|+\alpha+1)}.
\end{aligned}
\end{equation}
\end{proof}

\begin{corollary}\label{convergence of formal series}
The formal sum \eqref{f} converges in $L^2_{\alpha}(\Omega)$ when $\alpha >-1$.
\end{corollary}

\begin{proof}
The partial sums
\begin{equation}\label{partial sums}
F_{N+m} : = \sum_{|I|= 0}^{N+m} f_{I} (z) (T_z w)^I
\end{equation} satisfy
\begin{equation}\nonumber
\begin{aligned}
\| F_{N+m} \|^2
&= \frac{\pi^{n}}{n!} \sum_{l=0}^m  ||\varphi_{N+l}||^2 \frac{(N+l)! \Gamma(n+\alpha+1) }{\Gamma(n+N+l+\alpha+1)}\\
&= \frac{\pi^n}{n!} \sum_{l=0}^m \frac{(N+l)! \Gamma(n+\alpha+1) }{\Gamma(n+N+l+\alpha+1)}
 \bigg( \prod_{j=1}^{l} \left(1+ \frac{n-1}{N+j} \right) \bigg)
   \frac{(2N-1)!}{ \{ (N-1)!  \}^2} \frac{ \{ (N+l-1)! \}^2} { (2N+l-1)!}  \frac{1}{l! } \| \psi \|^2
\\
&=  \frac{\pi^n}{n!} \frac{\Gamma(n+\alpha+1)\Gamma(N+1)}{\Gamma(N+n+\alpha+1)}
 \sum_{l=0}^{m} \frac{(N+1)_l}{(n+N+\alpha+1)_l} \frac{(N)_l (N)_l}{(2N)_l} \frac{1}{l!} \bigg(\prod_{j=1}^{\ell} \left(1+ \frac{n-1}{N+j} \right)  \bigg)
\end{aligned}
\end{equation}
by Lemma \ref{norm of f} and Lemma \ref{norm of varphi}
where $(N)_l = N(N+1)\cdots (N+l-1)$.

Now let
\begin{equation}\label{a_l}
a_l := \frac{(N+1)_l}{(n+N+\alpha+1)_l} \frac{(N)_l (N)_l}{(2N)_l} \frac{1}{l!} \prod_{j=1}^{\ell} \left(1+ \frac{n-1}{N+j} \right).
\end{equation}
Then
\begin{equation*}
\begin{aligned}
l\left(\frac{a_l}{a_{l+1}}-1 \right)
&= l\left( \frac{(l+1)(l+2N)(l+n+N+\alpha +1)}{(l+N)^2 (l+n+N)} -1 \right) \\
&= l\bigg(\frac{(l+1)(l+2N)}{(l+N)^2} \left( 1+ \frac{\alpha+1}{l+n+N} \right) -1 \bigg) \\
&= l\left( \frac{(l+1)(l+2N)}{(l+N)^2} -1 \right) + (\alpha+1) \frac{l (l+1)(l+2N)}{(l+N)^2 (l+n+N)}  \rightarrow 1+ (\alpha+1)
\end{aligned}
\end{equation*}
by letting $l \rightarrow \infty$.
Hence if $\alpha >-1$, the formal sum \eqref{f} converges in $L^2_{\alpha}(\Omega)$ as $m \rightarrow \infty$ by Raabe's test.
\end{proof}

\begin{lemma}
Let $f$ be the $L^2$-limit of the partial sums \eqref{partial sums} on $\Omega$. Then
$f$ is holomorphic on $\Omega$.
\end{lemma}

\begin{proof}
Let
$
F_m(z,w) := \sum_{|I|= 0}^m   f_{I}(z) (T_z w)^I
$
and
$
\varphi_{I} := f_{I}  e^I
$.
It suffices to show that
\begin{equation}
\| \bar \partial F_{m} \|^{2}_{\alpha} \rightarrow 0
\end{equation}
as $m \rightarrow \infty$ for $\alpha=1$.

Since
$$
\frac{\partial F_{m}}{\partial \bar z_j } (z,w)
= \frac{\partial \tilde F_m}{\partial \bar z_j} (z,T_zw)
+ \sum_{k} \frac{\partial \tilde F_m} {\partial t_k} (z,T_zw) \frac{\partial (T_z w)_{k}} {\partial \bar z_j}
$$
with $\tilde F_m(z,t) := \sum_{|I|= 0}^m   f_I(z) t^I$
by the similar way as we induced the equation \eqref{necessary to be holomorphic}, we obtain
\begin{equation}\nonumber
\begin{aligned}
 &\overline X_{\mu} F_m = \sum_{j} \overline  A^{j \mu} \left(  \frac{\partial \tilde F_m}{\partial \bar z_j} (z,T_zw)
+ \sum_{k} \frac{\partial \tilde F_m} {\partial t_k} (z,T_zw) \frac{\partial (T_z w)_{k}} {\partial \bar z_j} \right) \\
&= \sum_{|I|= 0}^m \left(
\overline X_\mu f_I
+ \sum_{j, k}  \bar A^{j\mu}i_kf_I
\left( B_k^{jk} + \sum_{s=1}^{k-1}  B_{ks}^{jk} (T_zw)_s
+ \sum_{s=k+1}^n B_{sk}^{jk}(T_zw)_s
+ B_{kk}^{jk} (T_zw)_k\right) \right)(T_zw)^I \\
&= \sum_{|I|= 0}^m\left(
\overline X_\mu f_I
+f_I  \sum_k i_k \Gamma_k^{\mu k} +
 |I| f_I\sum_s  (T_zw)_s
\right)(T_zw)^I \\
\end{aligned}
\end{equation}
by \eqref{BC1} and Lemma \ref{ABsum}.
If we express $\varphi_{l} =  \sum_{|I|=l} f_{I} e^{I}$, we have
\begin{equation}\label{dbar1}
\bar \partial \varphi_{l}
=  \sum_{|I| =l} \sum_{\mu} \left( \overline X_{\mu} f_{I} + \sum_{k} i_k f_{I} \Gamma_{k}^{\mu k} \right) e^I \otimes \bar e_{\mu}.
\end{equation}
On the other hand, one has
\begin{equation}\label{dbar2}
\begin{aligned}
\bar \partial \varphi_{l}
= -(l-1) \mathcal R_G( \varphi_{l-1} )
= -(l-1) \sum_{\mu=1}^{n}   \sum_{|J| = l-1}  f_{J} e^J e_{\mu} \otimes \bar e_{\mu}.
\end{aligned}
\end{equation}
Hence by comparing \eqref{dbar1} and \eqref{dbar2} one obtains
$$
\sum_{|I| =l} \sum_{\mu} \left( \overline X_{\mu} f_{I} + \sum_{k} i_k f_{I} \Gamma_{k}^{\mu k} \right) t^I
= -(l-1)\sum_{\mu=1}^{n}   \sum_{|J| = l-1}  f_{J} t^J t_{\mu}.
$$
Therefore we obtain
$$
\overline X_\mu F_m = \sum_{|I| = m} m f_I (T_zw)^I\sum_k (T_zw)_k.
$$
If $g$ and $h$ are monomials in $t$ with $g\neq ch$ for any $c\in \mathbb R$, we have $\int_{\mathbb{B}^n}
g \bar h (1-|t|^2)^{\alpha} d \lambda_{t}=0$.
Hence by Lemma \ref{norm of varphi} one obtains
\begin{equation}\nonumber
\begin{aligned}
\| \bar \partial F_{m} \|^2_{1}
& = m^2 \bigg\| \sum_{|I| =m} \sum_{k} f_{I} (T_zw)^I (T_zw)_k \bigg\|^2_{1}\\
&=c_1 m^2 \sum_{|I| =m} \int_{\Sigma} \frac{|f_{I} (z)|^2}{ (1-|z|^2)^{n+1}} d \lambda_z
\bigg( \sum_{k} \int_{\mathbb{B}^n} | t^I t_k|^2 (1-|t|^2) d\lambda_{t} \bigg) \\
&=   m^2 \| \varphi_{m} \|^2
 \frac{  \pi^{n} \Gamma(n+2)(m+1)!}{n!\Gamma(n+m+3)}
\lesssim m^2 \frac{ \Gamma(n+2)(m+1)!}{\Gamma(n+m+3)}  \| \varphi_{m} \|^2\\
& \lesssim \frac{1}{(\frac{n}{m}+1+\frac{2}{m}) (\frac{n}{m} +1 + \frac{1}{m})} \frac{(m+1)!}{(m+n)!}  \| \varphi_m \|^2.
\end{aligned}
\end{equation}
Since
\begin{equation}
 \|\varphi_{N+m} \|^2  =\bigg( \prod_{j=1}^{m} \left(1+ \frac{n-1}{N+j} \right) \bigg) \frac{(2N-1)! \{ (N+m-1)!\}^2 }{\{(N-1)!\}^2 m! (2N+m-1)!  }   \| \psi \|^2 = O(m^{n-2})
\end{equation}
by the Stirling's formula and $\prod_{j=1}^{m} (1+ \frac{n-1}{N+j}) = O(m^{n-1})$,
\begin{equation*}
\frac{1}{(\frac{n}{m}+1+\frac{2}{m}) (\frac{n}{m} +1 + \frac{1}{m})} \frac{(m+1)!}{(m+n)!}  \| \varphi_m \|^2 = O(m^{-1}).
\end{equation*}
Therefore $\| \partial F_{m} \|^2_{1} \rightarrow 0$ as $m \rightarrow \infty$ and the lemma is proved.
\end{proof}

Now define $\Phi\colon \bigoplus_{k=n+2}^\infty H^0(\Sigma, S^kT^*_\Sigma)\rightarrow \mathcal O(\Omega )$ so that $\Phi(\psi)$ to be a $L^2$
holomorphic function $\sum_{|I|=0}^\infty f_I(z) (T_zw)^I$ for $\psi\in \bigoplus_{k=n+2}^\infty H^0(\Sigma, S^kT^*_\Sigma)$.
The following lemma completes the proof of Theorem \ref{main}.
\begin{lemma}
The image of $\Phi$ is dense on $\mathcal O (\Omega)/\mathcal I_{D}^{n+1} $  for the compact open topology.
\end{lemma}

\begin{proof}
The proof is similar to that given in \cite{Adachi}. We give a sketch of the proof for the
completion of the article. First, we note that
$(\mathcal O (\Omega)/\mathcal I_{D}^{n+1} ) \cap L^2(\Omega)$
is a closed subspace in $A^2(\Omega)$.
Let $f$ be a function which is orthogonal to $\Phi\left(\bigoplus_{j=n+2}^\infty H^0(\Sigma, S^jT^*_\Sigma)\right)$
in $(\mathcal O (\Omega)/\mathcal I_{D}^{n+1} ) \cap L^2(\Omega)$
and let $\{\varphi_j\} \in C^\infty(\Sigma, S^{m}T_{\Sigma}^{*})$  be its associated differentials.
Let $N \geq n+2$ be the least integer such that $\varphi_{j} = 0$ if $j <N$, but $\varphi_{N}\not = 0$.
In particular $\varphi_N \in H^0(\Sigma, S^NT^*_\Sigma)$.
Let  $$
\Pi^{0}_{N+m, E_{N+m}} : L^2(\Sigma, S^{N+m}T_{\Sigma}^{*}\otimes \Lambda^{j} T_{\Sigma}^{*} )\rightarrow \text{Ker}(\Box^{(0)} - E_{N+m} I),
$$
be the projection. Then $\Phi(\varphi_{N})$ is spanned by $\{ \Pi^{0}_{N+m,E_{N+m}} (\varphi_{m}) \}$, $m \geq 0$
since we choose the $L^2$-minimal solution to construct $\Phi$ and hence
$\langle\langle f, \Phi(\varphi_N) \rangle\rangle \gtrsim \langle\langle \varphi_N , \varphi_N \rangle\rangle_{0} >0$.
This  contradiction induces the lemma.
\end{proof}

\begin{corollary}\label{alpha<-1}
Let $\Gamma$ be a torsion-free cocompact lattice in $\text{Aut}(\mathbb{B}^n)$
and $\Sigma = \mathbb{B}^n / \Gamma$ be a compact complex hyperbolic form. Assume that $H^0 (\Sigma, S^m T_{\Sigma}^{*} ) =0$ for every $0< m \leq n+1$, then 
\begin{equation*}
A^2_{\alpha} \cong \mathbb{C} \quad \text{ for any } -(n+1) < \alpha \leq -1.
\end{equation*}
\end{corollary}

\begin{proof}
Suppose that there exists a nonconstant holomorphic function $f \in A^2_{\alpha} (\Omega)$. Let $\{\varphi_m\}$ be the associated differential of $f$. By the assumption that $H^{0}(\Sigma, S^{m}T_{\Sigma}^{*})=0$ for every $0 < m \leq n+1$, there exists $N\in \mathbb N$ with $N \geq n+2$ such that $\varphi_{m} \equiv 0$ when $m <N$, but $\varphi_{N}\not \equiv 0$. Note that $ \psi:= \varphi_{N}$ is a  symmetric differential on $\Sigma$.  Let $\{ \varphi'_m \}$ be the sequence defined by (\ref{system}) and (\ref{system2}) with $\varphi '_{N} := \varphi_N$.
Now we consider the orthogonal projection operator  $ \Pi_{N+m, E_{N,m}}\colon L^2(\Sigma, S^{N+m}T_{\Sigma}^{*}) \rightarrow \text{Ker}(\Box_{N+m}^{0} - E_{N,m} I)$ where  $E_{N,m}$ is given in the proof of Lemma \ref{norm of varphi}.  Note that $\text{Ker}(\Box^{0} - E_{N,m})$ is not trivial if $N \geq n+2$. By the same argument of Proposition 6.3 \cite{Adachi}, we know $\|\varphi_{N+m} \|^2 \geq  \| \Pi_{N+m, E_{N,m} } \varphi_{N+m} \|^2 = \| \varphi'_{N+m} \|^2 $. Therefore by the proof of Corollary \ref{convergence of formal series}, we obtain
\begin{equation*}
\begin{aligned}
\| f \|^2_{\alpha} &= \frac{ \pi^{n}}{n!} \sum_{|I|=0}^\infty ||\varphi_{|I|}||^2
\frac{ |I|! \Gamma(n+\alpha+1) }{\Gamma(n+|I|+\alpha+1)}\\
& \geq \frac{\pi^n}{n!} \sum_{l=0}^{\infty} \| \Pi_{N+l, E_{N,l} } \varphi_{N+l}    \|^2\frac{ (N+l)! \Gamma(n+\alpha+1) }{\Gamma(n+N+l+\alpha+1)} \\
& = \frac{\pi^n} {n!}\frac{\Gamma(n+\alpha+1)\Gamma(N+1)}{\Gamma(N+n+\alpha+1)} \sum_{l=0}^{\infty} a_l
\end{aligned}
\end{equation*}
where $a_l$ is given in \eqref{a_l}.
Since
\begin{equation*}
\begin{aligned}
\frac{a_l}{a_{l+1}} &= 1+ \left( \frac{(l+1)(l+2N)}{(l+N)^2} -1 \right) + (\alpha+1) \frac{ (l+1)(l+2N)}{(l+N)^2 (l+n+N)} \\
&= 1+ \left( \frac{l+2N-N^2}{(l+N)^2} \right) + \frac{(\alpha+1)}{(l+N)^2}\left( l+N+(1-n) + \frac{n(n-1)-N(N-1)}{l+n+N}\right) \\
&= 1+\frac{l+2N-N^2 + (\alpha+1) (l+N+(1-n))}{(l+N)^2}+ O\left(\frac{1}{l^3}\right) \\
&= 1+\frac{l(1+(\alpha+1))}{(l+N)^2} + O\left(\frac{1}{l^2} \right) \\
&= 1+\frac{1+(\alpha+1)}{l} \frac{1}{(1+\frac{N}{l})^2} + O \left(\frac{1}{l^2} \right) \\
&= 1+ \frac{1+(\alpha+1)}{l} \left( \sum_{k=0}^{\infty} \bigg( -\frac{N}{l} \bigg)^k \right)^2 + O \left(\frac{1}{l^2} \right) \\
&= 1+ \frac{1+(\alpha+1)}{l} + O\left( \frac{1}{l^2} \right),
\end{aligned}
\end{equation*}
if $1+(\alpha+1) \leq 1$, then $\sum_{l=0}^\infty a_{l}$ diverges by the Gauss test. Therefore, $A^2_{\alpha} (\Omega) \cong \mathbb{C}$ when $\alpha \leq -1$.
\end{proof}

\begin{theorem}
Let $\mathbb B^n/\Gamma$ be a compact complex hyperbolic space where $\Gamma$ is a torsion-free cocompact discrete subgroup of the automorphism group of $\mathbb B^n$. Assume that
$H^0 (\Sigma, S^m T_{\Sigma}^{*})=0$ for every $0<m \leq n+1$, then
$\Omega = \mathbb B^n \times\mathbb B^n /\Gamma$ has no nonconstant bounded holomorphic function.
\end{theorem}
\begin{proof}
Let $f$ be a bounded holomorphic function
on $\Omega$. Since
\begin{equation*}
\begin{aligned}
\|f \|^2_{\alpha}
&=  c_{\alpha} \int_{\Sigma} \frac{1}{(1-|z|^2)^{n+1}} d\lambda_{z}\int_{\mathbb{B}^n} \bigg|\sum_{|I|=0}^{\infty} f_{I}(z) t^I \bigg|^2 (1-|t|^2)^{\alpha} d \lambda_{t} \\
&\leq c_{\alpha} \big( \sup_\Omega |f| \big)^2  \int_{\Sigma} \frac{1}{(1-|z|^2)^{n+1}} d\lambda_{z}\int_{\mathbb{B}^n} (1-|t|^2)^{\alpha} d \lambda_{t}\\
&\lesssim \big(\sup_\Omega |f| \big)^2 \text{Vol}(\Sigma),
\end{aligned}
\end{equation*}
 $\|f\|^2_{\alpha}$ is uniformly bounded for all $\alpha>-1$.
Hence we have $$
\lim_{\alpha \searrow -1} \| f \|^2_{\alpha} = \frac{ \pi^{n}}{n!}  \sum_{|I|=0}^\infty ||\varphi_{|I|}||^2 \frac{ |I|! \Gamma(n) }{\Gamma(n+|I|)} = \|f \|^2_{-1}$$
and as a result $f\in A^2_{-1} (\Omega)$.
However by Corollary \ref{alpha<-1} it is a contradiction.
\end{proof}


\begin{thebibliography} {XXX}
\bibitem{Adachi}
Adachi, Masanori {\it Weighted Bergman spaces of domains with Levi-flat boundary: geodesic segments on compact Riemann surfaces}.
arXiv:1703.08165


\bibitem{Brunebarbe-Klingler-Totaro}
Brunebarbe, Yohan; Klingler, Bruno; Totaro, Burt Symmetric differentials and the fundamental group. Duke Math. J.  162  (2013),  no. 14, 2797--2813.

\bibitem{Hwang_To}
Hwang, Jun-Muk; To, Wing-Keung {\it Syzygies of compact complex hyperbolic manifolds}. J. Algebraic Geom.  22  (2013),  no. 1, 175--200. 

\bibitem{Klingler}
Klingler, Bruno {\it Symmetric differentials, K\"ahler groups and ball quotients}.
Invent. Math.  192  (2013),  no. 2, 257--286.

\bibitem{Mok}
Mok, Ngaiming {\it Metric rigidity theorems on Hermitian locally symmetric manifolds}. Series in Pure Mathematics, 6. World Scientific Publishing Co., Inc., Teaneck, NJ, 1989. xiv+278 pp. ISBN: 9971-50-800-1; 9971-50-802-8


\bibitem{Seo}
Seo, Aeryeong {\it Weakly 1-completeness of holomorphic fiber bundles over compact K\"ahler manifolds}.
arXiv:1907.11151

\bibitem{Wong}
Wong, Kwok-Kin {\it On effective existence of symmetric differentials of complex hyperbolic space forms}.
Math. Z.  290  (2018),  no. 3--4


\bibitem{Zhu}
Zhu, Kehe {\it Spaces of holomorphic functions in the unit ball}.
Graduate Texts in Mathematics, 226. Springer-Verlag, New York, 2005. x+271 pp. ISBN: 0-387-22036-4


\end{thebibliography}
\end{document}